\documentclass[11pt]{amsart}
\usepackage{mathtools, amssymb, amsfonts, amsthm}
\usepackage{epsfig, psfrag, color, multicol, graphicx, enumitem, pdfsync, bm}
\usepackage{amsfonts}
\usepackage{epstopdf}

\usepackage{tikz-cd}
\newtheorem{thm}{Theorem}[section]

\newtheorem{cor}[thm]{Corollary}
\newtheorem{prop}[thm]{Proposition}

\newtheorem{conj}[thm]{Conjecture}

\theoremstyle{plain}
\newtheorem{theo}[thm]{Theorem}
\newtheorem{lem}[thm]{Lemma}

\theoremstyle{definition}
\newtheorem{defi}[thm]{Definition}
\newtheorem{eg}[thm]{Example}
\newtheorem{rem}[thm]{Remark}

\numberwithin{equation}{section}

\def\zz{\mathbb Z}
\def\nn{\mathbb N}

\def\rr{\mathbb R}

\def\ov{\overline}

\def\la{\lambda}
\def\ga{\gamma}

\def\ep{\epsilon}
\def\al{\alpha}
\def\be{\beta}

\def\ssu{\subset}

\def\wt{\widetilde}
\def\<{\langle}
\def\>{\rangle}

\def\rq{ {\text {\rm q}  } }

\def\Z{ {\text {\rm Z} } }

\def\vt{\vartheta}

\def\Q{{\text {\rm Q} } }

\def\0{{\mathbf 0}}

\def\NN{{\mathbb N}}

\def\.{\hskip.06cm}
\def\ts{\hskip.03cm}

\def\ind{{\text {\rm ind}}}

\def\bx{{\textbf{x}}}
\def\bt{{\textbf{t}}}

\def\rCr{{s}}

\def\poly{\textup{\textsf{P}}}
\def\po{\textup{\textsf{/poly}}}
\def\Pp{{\textup{\textsf{P}}}}
\def\FP{{\textup{\textsf{FP}}}}
\def\SP{{\textup{\textsf{\#P}}}}
\def\FP{{\textup{\textsf{FP}}}}
\def\NP{{\textup{\textsf{NP}}}}

\def\coNP{{\textup{\textsf{coNP}}}}
\def\Ppo{{\textup{\textsf{P/poly}}}}
\def\FPpo{{\textup{\textsf{FP/poly}}}}
\def\SigmaP{\boldsymbol{\Sigma}^{\poly}}
\def\PiP{\boldsymbol{\Pi}^{\poly}}
\def\sharpP{\textup{\textsf{\#P}}}
\def\PH{\textup{\textsf{PH}}}

\def\BPP{\textup{\textsf{BPP}}}
\def\unique{\textup{\textsf{U}}}

\def\GH{\textup{\textsf{GH}}}
\def\GPA{\textup{\textsf{GPA}}}
\def\Gzero{\textup{\textsf{G}}}
\def\SigmaG{\boldsymbol{\Sigma}^{\Gzero}}
\def\SigmaPA{\boldsymbol{\Sigma}^{\textsf{\textup{PA}}}}
\def\PiG{\boldsymbol{\Pi}^{\Gzero}}
\def\PiPA{\boldsymbol{\Pi}^{\textsf{\textup{PA}}}}

\def\GF{\mathcal{GF}}
\def\GFstar{\GF^{\.*}}

\renewcommand{\ind}[1]{\text{index}(#1)}
\def\supp{\textup{supp}}
\def\proj{\textup{proj}}
\def\cproj{\cj{\textup{proj}}}

\renewcommand\L{\mathcal{L}}
\def\spec{\textup{spec}}

\def\Z{\mathbb{Z}}
\def\R{\mathbb{R}}
\def\N{\mathbb{N}}
\def\Q{\mathbb{Q}}

\newcommand{\cj}[1]{\overline{#1}}
\newcommand{\n}{\cj{n}}
\renewcommand{\b}{\cj{b}}

\def\a{\cj{a}}
\def\c{\cj{c}}
\def\d{\cj{d}}

\newcommand{\x}{\mathbf{x}}
\renewcommand{\t}{\mathbf{t}}
\renewcommand{\u}{\mathbf{u}}
\def\zzeta{\boldsymbol{\zeta}}
\def\xxi{\boldsymbol{\xi}}

\newcommand{\y}{\mathbf{y}}
\newcommand{\z}{\mathbf{z}}

\newcommand{\floor}[1]{\lfloor#1\rfloor}

\def\polyin{\textup{poly}}

\newcommand{\ex}{\exists}
\renewcommand{\for}{\forall}

\def\tauHad{\, \star_{\tau} \,}

\def\nin{\noindent}

\def\Pr{\textup{PA}}

\newcommand{\cpl}{\backslash}
\def\M{\mathcal{M}}
\def\v{\mathbf{v}}

\def\AP{\textup{AP}}
\def\w{\mathbf{w}}

\def\phi{\ell}
\def\GFof{{\bf F}}

\def\ellR{{r}}

\title{Complexity of short generating functions}

\author[Danny Nguyen \and Igor Pak]{Danny Nguyen$^{\star}$ \and Igor~Pak$^{\star}$}

\thanks{\thinspace ${\hspace{-.45ex}}^\star$Department of Mathematics,
UCLA, Los Angeles, CA, 90095.
\hskip.06cm
Email:
\hskip.06cm
\texttt{\{ldnguyen,\ts{pak}\}@math.ucla.edu}}

\thanks{\today}

\begin{document}
\maketitle

% {\begin{center}
% \today
% \end{center}
% }

\begin{abstract}
We give complexity analysis of the class of \emph{short generating functions} (GF).
Assuming \ts \sharpP$\ts\not\subseteq$\ts\FPpo, we show that this class is
not closed under taking many intersections, unions or projections of GFs,
in the sense that these operations can increase the bit length of coefficients
of GFs by a super-polynomial factor.  We also prove that
\emph{truncated theta functions} are hard in this class.
\end{abstract}

\vskip1.5cm

\section{Introduction}

\subsection{Combinatorics and complexity of GFs}
A \emph{short generating function} (short GF) is a rational generating function written in the form
$$(\ast) \qquad f(t) \, = \,
\sum_{i=1}^M \, \frac{c_{i} \. t^{a_i}}{(1-t^{b_{i\ts 1}})\cdots (1-t^{b_{i\ts k_{i}}})}\.,
$$
where $c_{i} = p_{i}/q_{i} \in \Q ,\; a_{i},b_{ij} \in \Z$ and $b_{ij} \neq 0$ for all $i,j$.
The  \ts index$(f) \coloneqq \max \{k_1,\ldots,k_M\}$ \ts is the maximum number of terms in the denominators. This is always assumed to be bounded by some constant.
The \emph{length} $\phi(f)$ is defined as the total bit lengths
of all constants in $(*)$.  Of course, the same generating function can have many presentations
as a short GF.\footnote{We also caution the reader that in general, the word
\emph{short} in ``short GF'' only means that the GF
is given in the form ($\ast$).  It does not necessarily  mean the GF has polynomial length.}
% for example:
% $$(\diamond) \quad
% 1+ t+ t^2+ \ldots + t^8 \, = \, \frac{1}{1-t} \, - \, \frac{t^9}{1-t}
% $$
%
% t+t^2+t^4+t^7 \, = \, \frac{1}{1-t} \, - \, \frac{1}{(1-t^3)(1-t^5)}  \, - \, \frac{t^{15}}{(1-t^3)(1-t^5)}\,.
%\, = \, \frac{1}{1-t^2} - \frac{t^{4}}{1-t^2} + \frac{t}{1-t^3} - \frac{t^{10}}{1-t^3}

In this paper we initiate the study of complexity of short GFs with bounded index and
polynomial lengths.
For a finite set $S \ssu \nn$, denote by $f_S(t) = \sum_{n\in S} t^n$ the GF of~$S$.
We are interested in deciding if it is possible to write $f_S$ as a short GF with
polynomial length for a variety of sets~$S$ coming from Combinatorics, Number Theory and
Discrete Geometry.
%
% Here we measure length of $f_{S}$ against some natural parameter of $S$, for example  $\max\{\log |n| : n \in S\}$.
%
% Danny, this sentence is more confusing than clarifying.
%
Showing that some sets do not have short GFs of polynomial lengths turns out to be a surprisingly difficult problem.
We are also interested in operations on short GFs and how they affect the short GFs' lengths.

%% For a finite set $S \ssu \nn$, denote by $f_S(t) = \sum_{n\in S} t^n$ the GF of~$S$.
%% We are interested in deciding if it is possible to write $f_S$ as a short GF with
%% polynomial length for a variety of sets~$S$ coming from Combinatorics, Number Theory and
%% Discrete Geometry.

% Note that if we also require $c_1=\ldots=c_M$ as in~$(\ast)$,
% this problem becomes a partition problem into (generalized) arithmetic progressions, which is closely
% related to many problems in Arithmetic Combinatorics (see~\cite{TV}).

%% In this paper we initiate the study of complexity of short GFs with bounded index and
%% polynomial lengths.  The problem which GFs are short in this sense turn out to be
%% surprisingly difficult.  We are especially interested in operations on short GFs
%% and how these operations affect the length of the short GFs.

Our approach is motivated by ideas from the study of integer points in convex polyhedra
in fixed dimension (see~$\S$\ref{ss:finrem-bar}).
All such polyhedra turn out to have (multivariate) short GFs of polynomial lengths
(see Definition~\ref{def:short GF multi} and Barvinok's Theorem~\ref{th:Barvinok} below).
%
% An important advance by Barvinok and Woods (Theorem~\ref{th:BW}) implies that the same holds for projections of convex polyhedra.
We refer to~\cite{B2,B3} for a thorough review of past and recent work on short~GFs
in Discrete Geometry, and to Section~\ref{sec:fin-rem} for connections to Arithmetic
Combinatorics and other areas.
% in connection to integer programming.
%
% For a general generating function, which is not necessarily written in the form $(\divideontimes)$, we call it a \emph{multi-variable GF} or simply just GF.
%%Throughout the paper, we consider only GFs with $0/1$ coefficients when expanded as a Laurent series, unless otherwise mentioned.

% We now state the main results of the paper, first for specific GFs, then for
% operations with one-variable short~GFs, and then in the multivariate case.

\subsection{Squares}
Define the \emph{truncated theta function} to be the GF over squares $\le 2^\ellR$~:
$$\vt_\ellR(t) \, = \, \sum_{n=0}^{2^{\ellR/2}} \. t^{n^2}\,.
$$

\begin{conj}[= Conjecture~\ref{conj:squares_long}]  \label{conj:squares-intro}
For every fixed $k\ge 1$, the truncated theta function $\vt_\ellR(t)$ \emph{cannot} be written a short GF
of length \. $\polyin(\ellR)$ \ts and \ts {\rm index}$(\vt_\ellR)\le k$.
\end{conj}

The following result is the most surprising result of this paper:

\begin{theo}[= Theorem~\ref{t:squares-sharp}]
If \.\ts $\sharpP \not\subseteq\FPpo$, \ts then Conjecture~\ref{conj:squares-intro} holds.
\end{theo}

In other words, if each truncated theta function can be represented as a short GF of polynomial length and bounded index, then any counting problem can be solved with polynomial size circuits.
See~$\S$\ref{ss:finrem-factoring} for more on the complexity assumption, and
Section~\ref{sec:squares-primes} for the related results on primes.
% and near-square primes.

\subsection{One variable operations}

%% For the rest of the paper we consider only GF of finite or infinite sets, i.e.\
%% power series with coefficients in $\{0,1\}$ defined by their support $\supp(f)\ssu \nn^d$.
Recall that we only consider GFs of finite sets.
We define operations on GFs based on their supports.
For example, taking the union of two GFs $f(t)$ and $g(t)$ means
finding another GF $h(t)$ with $\supp(h) = \supp(f) \cup \supp(g)$.
We can similarly define other Boolean operations.

Short GFs are known to be very versatile and useful in applications.
%flexible objects compared to other forms of GFs (e.g.\ algebraic).
Notably, given a bounded number of short GFs, all Boolean operations
on them can be performed in polynomial time (see~\cite{B3,BP}).
The result is again a short GF with polynomial length.
However, when the number of short GFs is large, no such polynomial time procedures are known.
The following result gives a strong evidence against such possibility:
%% a similar result for intersection/union of possibly many short
%% GFs (instead of a bounded number).

\begin{theo}[=Theorem~\ref{th:union_long}]\label{th:main_2}
  If \.\ts $\sharpP \not\subseteq\FPpo$, \ts  then
taking intersection/union of many short GFs does not preserve polynomiality in length.
%there exist many short GFs of polynomial total length whose intersection cannot be written as short GFs of polynomial length.
\end{theo}

This says taking union of many short GFs is hard structurally.
It should be compared to an earlier result by Woods, which says that taking union of many short GFs is hard algorithmically, assuming $\poly \ne \NP$ (see Theorem~\ref{th:GF_NP_hard} and the following remark).

Next, define the \emph{Minkowski sum} $f \oplus g$ of two GFs $f(t)$ and $g(t)$,
to be the GF $h(t)$ with  $\ts \supp(h) = \supp(f) \oplus\supp(g) = \{a+b \mid a\in \supp(f), b \in \supp(g)\}$.

\begin{theo}[=Theorem~\ref{th:Minkowski_long}]\label{th:main_3}
  If \.\ts $\sharpP \not\subseteq\FPpo$, \ts  then
taking Minkowski sum of \emph{two} short GFs does not preserve polynomiality in length.
\end{theo}

Giving precise formulations of these results requires some effort, see
Section~\ref{sec:int-unions}.
Let us mention that in both theorems we can substitute the complexity assumptions with
Conjecture~\ref{conj:squares-intro}.
These results show strong limitations of the ``short GF technology'' from a geometric
point of view (see~$\S$\ref{ss:finrem-bar}). Below we give further evidence of this phenomenon.

\subsection{Projections}
For multivariate short GFs, taking projections is a key operation.
Projection is crucial for applications such as Integer Programming (see e.g.~\cite{E2,K2,NP}),
and theoretical considerations such as Presburger Arithmetic (see e.g.~\cite{B2,NP1,W} and~$\S$\ref{sec:pres}).
In a crucial development, Barvinok and Woods~\cite{BW} showed that given a polytope $P$ in bounded dimension,
the projections of its integer points on some subspace have a short GF of polynomial length,
which can also be computed in polynomial time (Theorem~\ref{th:BW}).
This result exploited the polytopal structure of $P$ and its convexity in a crucial way.
Unfortunately, these are also the reasons that prevent their result to apply on a non-geometric level.
In other words, the algorithm by Barvinok and Woods cannot produce a short~GF
for the projections if the input is presented only as short~GF,
without a polytope associated to it.

An important negative result by Woods in fact shows
that given only a multivariate short GF $f(\t)$, computing its projection is $\coNP$-hard
(see Theorem~\ref{th:GF_NP_hard} and the Remark~\ref{rem:Woods-NP}).
The following theorem is the central result of the paper.  Roughly speaking, it both
weakens the assumptions and strengthens the conclusions of Woods's theorem.

\begin{theo}[=Corollary~\ref{cor:proj_long_strong}]\label{th:main_1}
If \.\ts $\sharpP \not\subseteq\FPpo$, \ts  then
taking projection of a short GF does not preserve polynomiality in length.
\end{theo}

This says that in general not only we cannot \emph{compute} the projection of a short~GF
in polynomial time, %%some projections of short~GFs have projections so that
%%every short~GF presentation
any short GF that represents the projection
must have a super-polynomial length.
In other words, the barriers of using the ``short GF technology'' in this case
are structural rather than algorithmic.
%% Again this should be compared to the earlier result by Woods, which says that projecting short GF is hard algorithmically, assuming $\poly \ne \NP$ (see Theorem~\ref{th:GF_NP_hard} and the following remark).

The next result can be viewed as a refinement of the previous theorem, giving
a precise characterization of complexity of projections.

\begin{theo}[=Theorem~\ref{th:GH_PHpo}]\label{th:main_4}
Repeated projections of short GFs can encode every language in the non-uniform
polynomial hierarchy $\PH\po$. In fact, they form a hierarchy that
coincides with~$\PH\po$.
\end{theo}

We postpone the precise formulations of these results, especially of
Theorem~\ref{th:main_4} where the technicalities are unavoidable.
Let us also mention Proposition~\ref{prop:partial_converse} which
can be viewed as a partial converse of Theorem~\ref{th:main_1}
(cf.~$\S$\ref{t:squares-sharp}).\footnote{By itself,
Conjecture~\ref{conj:squares-intro} does not necessarily imply that
\ts\ts $\sharpP \not\subseteq\FPpo$, so a stronger assumption is used
in Proposition~\ref{prop:partial_converse}.}

\subsection{Paper structure}
The results in this paper are largely self-contained and require little more
than a few technical lemmas from~\cite{BP}, which are all stated in Section~\ref{sec:operations}
and can be treated as black boxes. We do however employ a fair amount of definitions and
notations (sections~\ref{sec:notations} and~\ref{sec:operations}).  We also assume
the reader is familiar with basic Computational Complexity, which goes to the heart
of this paper.  We refer the reader to~\cite{MM,Pap} for the standard results and notation,
and to~\cite{Aar} for a comprehensive recent survey.

Our Section~\ref{sec:short-GF-complexity} is the key as it describes the connection
between languages and short GFs.  From this point on, the reader can proceed
to the development of the short GF hierarchy, culminating in the proofs of theorems~\ref{th:main_1}
and~\ref{th:main_4} (sections~\ref{sec:nu-hier}--\ref{sec:long-proj}).
Alternatively, modulo a few definitions in earlier section, the reader proceed directly
to the proof of theorems~\ref{th:main_2} and~\ref{th:main_3} in Section~\ref{sec:int-unions}.
Similarly, the reader can also proceed to study complexity of squares
and primes (Section~\ref{sec:squares-primes}). In Section~\ref{sec:rel} we investigate
more technical questions on relative complexity of short~GFs, and in Section~\ref{s:lemma-proof}
we give a proof of a technical Lemma~\ref{lem:compress}.  We conclude with final remarks
and open problems in Section~\ref{sec:fin-rem}.

\bigskip

\section{Notations} \label{sec:notations}

\nin
We use $\nn \ts = \ts \{0,1,2,\ldots\}$.

\nin
All constant vectors are denoted as $\a, \b, \c, \d, \n,$ etc. The all $1$ vector is also denoted by $1$.

%% \nin
%% The $L_{\infty}$ norm of $\a$ is denoted by $|\a|$.

\nin
Matrices are denoted as $A, B, C$, etc.

\nin
Single variables are denoted as $x,y,z$, etc.; vectors of variables are denoted as $\x, \y, \z$, etc.

\nin
We write $\x \le \y$ if $x_{j} \le y_{j}$ for all $i$.

\nin
For two tuples $\x$ and $\t$ both of length $n$, we denote by $\t^{\x}$ the monomial $t_{1}^{x_{1}} \dots t_{n}^{x_{n}}$.

%%\nin
%%We also write $\x \le N$ to mean that each coordinate is~$\le N$.

%% \nin
%% Denote by $[-N,N]^{n}$ the set of all vectors $\x \in \Z^{n}$ with $|\x| \le N$.

%%If $x_{j} \le c$ for every index $j$ with $c$ a constant, we write $\x \le c$.

%%We use $\floor{.}$ to denote the floor function.

%%The the vector $\y$ with coordinates $y_{i} = \floor{x_{i}}$ is denoted by $\y = \floor{\x}$.

%%Single-variable generating functions are denoted by $f(t), g(u), h(v)$, etc.

\nin
GF is an abbreviation for ``\emph{generating function}.''

\nin
Single-variable GFs are denoted as $f(t), g(t), h(t), p(t), q(t)$, etc.

\nin
Multi-variable GFs are denoted as $f(\t), g(\t), h(\t), p(\t), q(\t)$, etc.

\nin
The support of a GF $f(\t)$ is denoted by $\supp(f)$.

\nin
The symbols $\lnot, \land$ and $\lor$ denote negation (complement), conjunction and disjunction.

\nin
A \emph{polyhedron} is an intersection of finitely many closed half-spaces in~$\rr^n$.

%% \emph{Copolyhedron/copolytope} is a polyhedron/polytope with possibly some open

%% facets.

\nin
A \emph{polytope} is a bounded polyhedron.

\nin
Polyhedra and polytopes are denoted as $P, Q, R$, etc.

\nin
The function $\phi(\cdot)$ denotes the \emph{bit length} of a number, vector, matrix, GF, or a logical formula when written in binary.

\nin
For a polyhedron $Q$ described by a linear system $A\x \le \b$, we denote by $\phi(Q)$ the total length
$\phi(A)+\phi(\b)$.

\bigskip

\section{Polynomial time operations on short GFs}  \label{sec:operations}

\subsection{Preliminaries on short GFs}\label{sec:prelim}
% For the definition of short GFs, please refer back to \eqref{???}.
%In this section, we consider short GFs of bounded indices, not necessarily coming from short Presburger formulas, and ask under what operations they remain so.
%% Again, we emphasize that all GFs considered have $0/1$ coefficients, unless otherwise mentioned.
%% Short GFs of the form $(\divideontimes)$ are always written binary for all results in this paper.

A power series $f(\t) = \sum \al_{\x} \t^{\x}$ is called a GF if each coefficient $\al_{\x}$ is either $0$ or $1$.
When needed, we will write $f(\t) = \sum \t^{\x}$ to emphasize that $f$ is a GF.

\begin{defi}\label{def:supp}
The support of an $n$-variable GF $g(\t) = \sum \t^{\x}$ is defined as:
\[
\supp(g) \coloneqq \{ \x \in \Z^{n} : [\t^{\x}] g(\t) = 1\}.
\]
Here $[\t^{\x}]$ denotes the coefficient of the monomial $\t^{\x}$ in $g(\t)$.
\end{defi}

\begin{defi}\label{def:proj_and_spec}
  Given a multi-variable GF $f(\t,\u) = \sum \t^{\x} \u^{\y}$ with $\x \in \Z^{m}, \y \in \Z^{n}$, the \emph{$\x$-projection} $g=\proj_{\x}(f)$ is the unique GF $g(\t) = \sum \t^{\x}$ with support satisfying
\[
\supp(g) = \{\x \in \Z^{m} : \ex \. \y\in\Z^{n} \;\; (\x, \y) \in \supp(f)\}.
\]
%%We write $\proj(f)$ without the subscript $\x$ whenever the context is clear.
If $f$ satisfies the extra property that for every $\x \in \Z^{m}$ there is at most one $\y \in \Z^{n}$ such that $(\x,\y) \in \supp(f)$, then $\proj_{\x} (f)$ is called the \emph{$\x$-specialization} of $f$, denoted by $\spec_{\x}(f)$.
%% or just $h = \spec(f)$ when the variable $\x$ is clear.
\end{defi}

\begin{defi}
  Consider two power series $f(\t) = \sum \al_{\x} \t^{\x}$ and $g(\t) = \sum \be_{\x} \t^{\x}$.
The \emph{Hadamard product} of $f$ and $g$, denoted by $f \star g$, is another GF $h(\t) = \sum \gamma_{\x} \t^{\x}$ with
  \[
  \gamma_{\x} = \al_{\x} \. \be_{\x} \; \text{ for every } \x.
\]
If $f$ and $g$ are GFs then the above condition is equivalent to $\supp(h) = \supp(f) \cap \supp(g)$.
\end{defi}

\begin{defi}  \label{def:short GF multi}
For a rational function in $n$ variables $\t = (t_{1},\dots,t_{n})$ of the form
\begin{equation*}
(\divideontimes) \qquad f(\t)  \, = \,
\sum_{i=1}^M \, \frac{c_{i} \. \bt^{\a_i}}{(1-\bt^{\b_{i\ts 1}})\cdots (1-\bt^{\b_{i\ts k_{i}}})}\ts,
\end{equation*}
the length $\phi(f)$ of $f$ is defined as
\begin{equation*}%%\label{eq:GF_length}
\phi(f) \, = \,
\sum_{i} \. \lceil\log_2 |p_{i}\. q_{i}|+1\rceil \, + \,
\sum_{i, j} \. \lceil\log_2 a_{i\ts j}+1\rceil \, + \,
\sum_{i,j,m} \. \lceil\log_2 b_{i \ts j \ts m}+1\rceil\ts,
\end{equation*}
%%where $c_{i} = p_{i}/q_{i} \in \Q$, $\a_i=(a_{i\ts 1},\ldots,a_{i\ts n}) \in \Z^{n}$ and $\b_{i\ts j}=(b_{i\ts j\ts 1},\ldots,b_{i\ts j\ts n}) \in \Z^{n}$.
where $c_{i}=p_{i}/q_{i} \in \Q ,\; \a_{i},\b_{ij} \in \Z^{n} ,\; \b_{ij} \neq 0$ and $\ts\bt^{\a} = t_1^{a_{1}}\cdots {}\ts t_n^{a_{n}}\ts$ if $\ts\a = (a_{1},\ldots,a_{n}) \in \zz^{n}$.
\end{defi}

\begin{defi}
For a power series $f(\t)=\sum \al_{\x} \t^{\x}$ given in the form $(\divideontimes)$, the \emph{index} of $f$ is defined as
\begin{equation*}
\text{index}(f) = \max\{k_{i} \;:\; i = 1,\dots,M\},
\end{equation*}
where $k_{i}$ is the number of factors in the denominator of the $i$-th summand.
\end{defi}

\begin{defi}\label{def:GF_def}
For every number of variables~$n$ and integer~$s$, we define two classes:
\begin{equation}\label{eq:GF_def}
\GF_{n,s} = \bigl\{ \text{GFs } g(\t) \text{ given in the form }(\divideontimes)\text{  with  } \ind{g} \le s   \bigr\}
\end{equation}
and
\begin{equation}
\GFstar_{n,s} = \bigl\{ \text{power series } g(\t) \text{ given in the form }(\divideontimes)\text{  with } \ind{g} \le s  \bigr\}.
\end{equation}
Members of $\GF_{n,s}$ are called \emph{short GFs}, while those of $\GFstar_{n,s}$ are called \emph{short power series}.
\end{defi}

We recall the following important results from~\cite{BP} (see also \cite{BW}):

\begin{theo}[\cite{BP}]\label{th:BPGF1}
Fix a class $\GF_{m,s}$. Given a short GF $f(\t) \in \GF_{m,s}$ of finite support.
We can compute in time $\polyin(\phi(f))$ the following:
\begin{enumerate}[label=\textup{\arabic*)}]
\item The norm \ts $N = \max \{|\x| : \x \in \supp(f) \}$,\footnote{Here $|\x|$ can be any polyhedral norm on $\x$, including $|\x|_{\infty}$ and $|\x|_{1}$.}
\item The cardinality $M = |\supp(f)|$, which is equal to $f(1)$,
\item The substitution $q(\u) = f(\t(\u))$, where $\t$ is substituted by monomials in some other variables $\u = (u_{1},\dots,u_{n})$.
Furthermore, we have $q(\u) \in \GFstar_{n,s}$.
\end{enumerate}
\end{theo}

\begin{theo}[\cite{BP}]\label{th:BPGF2}
Fix two classes $\;\GF_{m,s_{1}}$ and $\;\GF_{m,s_{2}}$.
Given $f(\t) \in  \GF_{m,s_{1}} $ and $\;g(\t) \in \GF_{m,s_{2}}$ of finite supports, we can compute in time $\polyin(\phi(f)+\phi(g))$ the following:
\begin{enumerate}[label=\textup{\arabic*)}]
\item A short GF $h(\t)$ with $\supp(h) = \supp(f) \cap \supp(g)$, i.e., $h(\t) = f(\t) \star g(\t)$,
\item A short GF $k(\t)$ with $\supp(k) = \supp(f)  \cup \supp(g)$.
 \item A short GF $p(\t)$ with $\supp(p) = \supp(f)  \cpl \supp(g)$.
 \end{enumerate}
 Moreover, we have $h,k,p \in \GF_{m,s_{1} + s_{2}}$.
 \end{theo}

 \begin{rem}\label{rem:Hadamard}
In fact, a more general version of Theorem~\ref{th:BPGF2} part 1) was shown in \cite{BP}, which also allows taking $f \star g$ for short power series.
\end{rem}

The following is the reason why we emphasized the bounded dimension~$n$ and index~$s$ in Definition~\ref{def:GF_def}.

\begin{prop}\label{prop:supp_is_P}
  Fix $n$ and $s$. Given a short power series $f(\t) = \sum \be_{\x} \t^{\x}$ in $\GF_{n,s}$ and a vector $\a_{0} \in \Z^{n}$, the coefficient $\beta_{\a_{0}}$ can be computed in time $\polyin(\phi(f) + \phi(\a_{0}))$.
\end{prop}

\begin{proof}
We let $g(\t) = \t^{\a_{0}}$ and define $h(\t) = f(\t) \, \star \, g(\t)$. Clearly, we have $h(\t) = \be_{\a_{0}} \. \t^{\a_{0}}$, which implies $\be_{\a_{0}} = h(1)$.
Applying Theorem~\ref{th:BPGF2}, we can compute $h(\t)$ (see also Remark~\ref{rem:Hadamard}). By Theorem~\ref{th:BPGF1}, we can compute $h(1)$. All can be done in time $\polyin(\phi(f) + \phi(\a_{0}))$.
\end{proof}

\begin{rem}
A similar result for $n$ and $s$ unbounded is unlikely to hold, considering the fact that $\textsc{KNAPSACK}$ is $\NP$-complete.
An instance of $\textsc{KNAPSACK}$ asks if an equation $a = \b\, \x$ is solvable, where $\x = (x_{1},\dots,x_{n}) \in \N$ are variables, and $a \in \N, \b \in \N^{n}$ are given as input.
This is equivalent to checking if $\.[t^{a}]f \neq 0$, where:
\[
f(t) \. = \. \frac{1}{(1-t^{b_{1}}) \cdots (1-t^{b_{n}})}.
\]
Here $n$ is not bounded.
Note that $\textsc{KNAPSACK}$ has a polynomial time algorithm if $a$ and $\b$ are given in unary.
In our case, short GFs are encoded in binary.
%%Note that the coefficients of $f$ are not $0/1$, but the notion of $\supp(f)$ is still well-defined.
\end{rem}

If $f$ is a short GF, Proposition~\ref{prop:supp_is_P} allows us to decide in polynomial time whether $\a_{0} \in \supp(f)$.
Now one may ask whether is it still easy to decide if a point $\a_{0}$ lies in a projection of $f$ . The answer is still positive:

\begin{prop}\label{prop:supp_proj_is_P}
Fix $m,n$ and $s$.
Given a short GF $f(\t,\u) = \sum \t^{\x} \u^{\y} \in \GF_{m+n,s}$ of finite support and a vector $\a_{0} \in \Z^{m}$, checking whether $\a_{0}\in\supp(\proj_{\x}(f))$ can be done in time $\polyin(\phi(f) + \phi(\a_{0}))$.
Here $\x \in \Z^{m}, \y \in \Z^{n}$.
\end{prop}
\begin{proof}
%%Let $f = \sum \t^{\x} \u^{\y}$ with $\x \in \Z^{m}$ and $\y \in \Z^{n}$.
%% Since $f$ has finite support, by Theorem~\ref{th:BPGF1} a), we can compute in polynomial time an $N$ for which $\supp(f) \in [-N,N]^{m+n}$ and $\log(N) = \polyin(\phi(f))$.
%% Define
%% \[
%% g(\t,\u) = \t^{\a_{0}} \Biggl( \sum_{\y \in [-N,N]^{n}} \u^{\y} \Biggr)   \; = \; \t^{\a_{0}} \, \prod_{i=1}^{n} \frac{u_{i}^{-N} - u_{i}^{N+1}}{1 - u_{i}},
%% \]
%% which is a short GF in $\.\GF_{m+n,n}\.$ of length $\polyin(\log(N) + \phi(\a_{0}))$.
%% We have $\a_{0} \in \supp(\proj_{\x}(f))$ if and only if there is a $\y \in [-N,N]^{n}$ for which $(\a_{0},\y) \in \supp(f)$.
%% This is equivalent to $f \star g \neq 0$, which again can be checked in polynomial time by Theorem~\ref{th:BPGF2}.
Let $g(\t) = f(\t, 1)$.
Clearly, we have $\a_{0} \in \supp(\proj_{\x}(f))$ if and only if the coefficient of $\,\t^{\a_{0}}$ in $g(\t)$ is non-zero.
By Theorem~\ref{th:BPGF1}, we can compute $g$ in time $\polyin(\phi(f))$.
By Proposition~\ref{prop:supp_is_P}, we can compute $[\t^{\a_{0}}]g$ in time $\polyin(\phi(g) + \phi(\a_{0})) \le \polyin(\phi(f) + \phi(\a_{0}))$.
\end{proof}

%% Given a short GF in $3$ variables $\x,\y$ and $\z$, one may ask if it is still easy to decide if a vector $\c_{0}$ lies in the support of $\proj_{\x}(\proj_{\x,\y}(f))$, i.e., $f$ projected twice.
%% Of course, this is still easy because two consecutive projections $\proj_{\x}$ and $\proj_{\x,\y}$ can be combined to a single projection $\proj_{\x}$ (with $\y$ and $\z$ projected simultaneously).
%% What happens if we insert a simple operation in between the two projections?
%% Then this becomes a very difficult problem (see Corollary~\ref{cor:2_proj_hard}).

In order to further study the projections of short GFs, we need a few logical tools.

\subsection{Presburger arithmetic and short GFs}\label{sec:pres}

\emph{Presburger Arithmetic} ($\Pr$) is the first order theory on the integers that allows only additions and inequalities. Each \emph{atom} (smallest term) in $\Pr$ is an integer inequality of the form
\[
a_{1} x_{1} + \ldots + a_{n} x_{n} \le b,
\]
where $\x = (x_{1},\dots,x_{n})$ are integer variables, and $a_{1},\dots,a_{n},b \in \Z$ are integer constants.
A general $\Pr$ formula is formed by taking Boolean combinations (negations, conjunctions, disjunctions) of such atoms, and also applying quantifiers ($\for/\ex$) over different variables. A sentence in $\Pr$ is a formula with all variables quantified.
The length~$\phi(F)$ of a $\Pr$ formula $F$ is the total length of all symbols and constants in~$F$ written in binary.

\begin{eg}
Let $P,Q \subseteq \R^{n}$ be two rational polyhedra given by two systems $A_{1} \x \le \b_{1}$ and $A_{2} \x \le \b_{2}$. Then the set of integer points in $P \cup Q$ is described by the $\Pr$ formula:
\[
F \, = \, \bigl\{ \x  : A_{1}\x \le \b_{1} \lor A_{2}\x \le \b_{2} \bigr\}.
\]
Here we are identifying the $\Pr$ formula $F$ with the set that it defines.
\end{eg}

\begin{eg}
{\rm
The $\Pr$ formula $\ts F = \bigl\{x \. : \. \for y \; \, (5y \ge x+1) \, \lor \,  (5y \le x - 1) \bigr\}$ determines the set of non-multiples of~$5$.
}
\end{eg}

\begin{defi}
  For a set $S \subseteq \zz^{n}$, denote by $\GFof(S; \t)$ the GF
  \begin{equation*}
\GFof(S;\t) \, = \, \sum_{\x \in S} \t^{\x}.
\end{equation*}
\end{defi}

$\Pr$ formulas are very well-suited to capture integer points in polyhedra.
The following cornerstone result by Barvinok says that integer points in a polyhedron in bounded dimension can be effectively enumerated by a short GF.

\begin{theo}[\cite{B1}]\label{th:Barvinok}
Fix $n$. Let $Q \subseteq \R^{n}$ be a rational polyhedron described by $A\x \le \b$.
There exists a short GF $f \in \GF_{n,n}$ with $\GFof(Q \cap \Z^{n};\t) = f(\t)$,
which can be computed in time $\polyin(\phi(Q))$.\footnote{This implies that $\phi(f) \le \polyin(\phi(Q))$.}
\end{theo}

We mention a useful tool about quantifier free $\Pr$ formulas:

\begin{prop}[{\cite[Prop.~5.2.2]{Woods}}]\label{prop:W}
Fix $n$. Let $\Phi(\x)$ be a Boolean combination of linear inequalities in integer variables $\x = (x_{1},\dots,x_{n})$. Then we have:
\begin{equation*}
\Phi(\x) = \textup{true} \quad \iff \quad \bigvee_{i=1}^{r} \x \in P_{i} \cap \Z^{n},
\end{equation*}
where $P_{1},\dots,P_{r} \subseteq \R^{n}$ are disjoint polyhedra and $r \le \polyin(\phi(\Phi))$.
The system defining each $P_{i}$ can be computed in time $\polyin(\phi(\Phi))$.
\end{prop}

Theorem~\ref{th:Barvinok} can be generalized to quantifier free $\Pr$ formula in bounded dimension:

\begin{theo}[{\cite[Prop.~5.3.1]{Woods}}]\label{th:W}
Fix $n$. Let $G = \{\x \in \Z^{n} : \Phi(\x)\}$ be a $\Pr$ formula with $\Phi$ a quantifier free Boolean combination of linear inequalities in $\x$.
There exists a short GF $g \in \GF_{n,n}$ with $\GFof(G;\t) = g(\t)$, which can be computed in time $\polyin(\phi(\Phi))$.
%%and satisfies $\phi(f) = \polyin(\phi(\Phi))$.
\end{theo}

\begin{proof}
By Proposition~\ref{prop:W}, we can rewrite $\Phi$ as a disjoint union of polyhedra $P_{1},\dots,P_{r}$ with $r \le \polyin(\phi(\Phi))$.
The system defining each $P_{i}$ can be computed in polynomial time.
Applying Theorem~\ref{th:Barvinok}, we get a short GF $f_{i} \in \GF_{n,n}$ of polynomial length for each $P_{i}$.
Summing up all $f_{i}$, we get a short GF $g \in \GF_{n,n}$ of length $\polyin(\phi(\Phi))$ for $G$.
\end{proof}

Next, we consider $\Pr$  formulas with quantifiers.
In the simplest case, $F$ encodes the projection of integer points in a polyhedron. For this, we have:

\begin{theo}[\cite{BW,NP}]\label{th:BW}
Fix $m, n \in \nn$.
Let $Q \subseteq \R^{m}$ be a rational polyhedron given by a system $A\x \le \b$, and $T : \Z^{m} \to \Z^{n}$ a linear map.
Consider the $\Pr$ formula
\[
G \. = \. \bigl\{\y \in \Z^{n} : \ex \, \x \in \Z^{m} \; (\x \in Q) \land (\y = T\x)\bigr\}\ts.
\]
Then there exists a short GF $g$ with $\GFof(G;\t) = g(\t)$, which can be computed in time $\polyin(\phi(Q) + \phi(T))$.
Furthermore, we have $g \in \GF_{n,s}$, where $s=s(m)$ is a constant.
\end{theo}

\begin{rem}
The above theorem was proved in \cite{BW} for the case when $P$ is a polytope.
It was recently extended  in~\cite{NP} to all (possibly unbounded) polyhedra.
\end{rem}

However, for general $\ex$-formulas, finding a short GF for $F$ becomes $\coNP$-hard:

\begin{theo}[{\cite[Th.~5.3.2]{Woods}}]\label{th:GF_NP_hard}
Let $\Phi(x,y)$ be a quantifier free Boolean combination of linear inequalities in $x$ and $y$ (singletons).
Consider
\[
F = \{y \in \Z  : \ex x \in \Z \;\; \Phi(x,y)\}.
\]
Then computing a short GF for F is $\coNP$-hard.
\end{theo}

\begin{rem}\label{rem:Woods-NP}
By Theorem~\ref{th:W}, we still can find a short GF of length $\polyin(\phi(\Phi))$ for $\Phi(x,y)$.
So this result says that projecting a short GF is hard algorithmically.
This should be compared to our Theorem~\ref{th:main_1}, which says that projecting short GF is hard structurally.
Actually, by Proposition~\ref{prop:W}, we can also decompose $\Phi(x,y)$ into a union of polynomially many polygons $P_{i} \subseteq \rr^{2}$.
By Theorem~\ref{th:BW}, the projection of integer points in each $P_{i}$ on $x$ has a short GF, which can be found in polynomial time.
So taking union of these short GFs is again hard algorithmically. This should be compared to Theorem~\ref{th:main_2}.
\end{rem}

\bigskip

\section{Short GFs and the class $\Ppo$}  \label{sec:short-GF-complexity}

\subsection{Encoding languages in $\Ppo$ as short GFs}
For technical reasons regarding the convergence of GFs under numerical evaluation, we consider only GFs with support in $\N^{n}$ from this section onwards.
Theorem~\ref{th:BPGF2} still applies to short GFs supported on $\N^{n}$.
%%Also all GFs are assumed to have $0/1$ coefficients.

\begin{defi}\label{def:L_l}
For every language $\L \in \{0,1\}^{*}$, and every $\ellR > 0$, we denote by $\L_{\ellR}$ the segment
\begin{equation}\label{eq:L_ell}
\L_{\ellR} \coloneqq \bigl\{ \wt x \in \{0,1\}^{\ellR} : \wt x \in \L \}.
\end{equation}
For $\wt x \in \L_{\ellR}$, let $x$ be the corresponding integer with binary representation $\wt x$.
We will also use $\L_{\ellR}$ to denote the set of all such $x$ with $\wt x \in \L_{\ellR}$.
\end{defi}

\begin{lem}\label{lem:Ppo_to_Pr}
For every language $\L \in \Ppo$, and every $\ellR > 0$, the segment $\L_{\ellR}$ can be characterized in $\Pr$ as:
\begin{equation}\label{eq:Ppo_Pr}
\wt x \in \L_{\ellR} \quad \iff \quad x \in [0,2^{\ellR})  \,\land\, \bigl[ \ex y \in [0,2^{p}) \; \for \z \in [0,2^{q})^3 : \Phi_{\ellR}(x,y,\z) \bigr],
%%\wt x \in \L_{\ellR} \quad \iff \quad \ex y   \; \for \z  \; \Phi_{\ellR}(x,y,\z),
\end{equation}
%%where $x$ is the integer $($in binary$)$ corresponding to the string $\wt x$,
%%\footnote{We assume that every $\wt x \in \L$ is written from left to right and must end with a $1$.
%%This corresponds to the least to most significant digits in $x$, with the most significant digit always $1$.
%%}
where $\Phi_{\ellR}$ is a quantifier free $\Pr$ expression in $x,y \in \N$ and $\z \in \N^{3}$.
Moreover, we have $p,q,\phi(\Phi_{\ellR}) \le \polyin_{\L}(\ellR)$.\footnote{We denote by $\phi(\Phi_{\ellR})$ the total length of all symbols $\Phi_{\ellR}$, written in binary. The notation $\polyin_{\L}(\ellR)$ denotes a polynomial in $\ellR$, with the polynomial degree depending on the language $\L$.}
If in addition $\L \in \poly$, then there is an algorithm to compute $p,q$ and $\Phi_{\ellR}$ in time $\polyin_{\L}(\ellR)$.
\end{lem}

\begin{proof}
By definition of the class $\Ppo$, there is a Boolean circuit $C_{\ellR}$ such that:
\begin{equation*}
\L_{\ellR} = \{\wt x \in \{0,1\}^{\ellR} \,:\, C_{\ellR}(\wt x) = \textup{true}\}.
\end{equation*}
Here the circuit $C_{\ellR}$ has $\ellR$ input gates, and as many as $p \le \textup{poly}_{\L}(\ellR)$ non-input gates, each with in-degree at most $2$. We encode the values of the non-input gates as a Boolean string $\wt y \in \{0,1\}^{p}$. Let $\wt x = (x_{1},\dots,x_{\ellR})$ and $\wt y = (y_{1},\dots,y_{p})$.
By a standard reduction (see e.g.\ \cite{MM,Pap}), we can encode the computation of $C_{\ellR}$ by a Boolean formula $F$ in $3$-Conjunctive Normal Form. Explicitly, we have:
\begin{equation}\label{eq:3SATset}
\L_{\ellR} = \{\wt x \in \{0,1\}^{\ellR} \,:\, \ex \wt y \in \{0,1\}^{p} \; F(\wt x,\wt y) = \text{true}\},
\end{equation}
where
\begin{equation}\label{eq:3SAT}
F(\wt x, \wt y) \;=\; \bigwedge_{k} (a_{k} \lor b_{k} \lor c_{k}).
\end{equation}
Here each $a_{k},b_{k},c_{k}$ is a literal in the set $\{x_{i}, \lnot x_{i},\, y_{j}, \lnot y_{j} \,:\, 1 \le i \le \ellR,\, 1 \le j \le p\}$.
\smallskip

Let $x \in [0,2^{\ellR})$ and $y \in [0,2^{p})$ be the integers corresponding to $\wt x$ and $\wt y$, respectively.
Every literal $x_{i}$ corresponds to the $i$-th digit in $x$ being $1$, and $\lnot x_{i}$ corresponds that digit being $0$.\footnote{The least significant digit in $x$ corresponds to $x_{0}$ in $\wt x$.}
In other words, $x_{i}$ is true or false respectively when $\floor{x/2^{i-1}}$ is odd or even.
The same applies to $y_{i}$ and $y$.
Observe that $t = \floor{x/2^{i-1}}$ is the only integer that satisfies $x/2^{i-1} - 1 < t \le x/2^{i-1}$.
Let $q = \max(\ellR,p) \le \polyin(\ellR)$.
Each term $x_{i}$ or $\lnot x_{i}$ can be coded with an extra $\ex z$ quantifier as follows:

\begin{equation}\label{eq:bit}
\aligned
x_{i} &\iff \ex z \in [0,2^{q}) :
\begin{Bmatrix}
2z+1 &> &x / 2^{i-1} - 1\\
2z+1 &\le &x / 2^{i-1}
\end{Bmatrix},
\\
\lnot x_{i} &\iff \ex z \in [0,2^{q}) :
\begin{Bmatrix}
2z &> &x / 2^{i-1} - 1\\
2z &\le &x / 2^{i-1}
\end{Bmatrix}
%%\quad , \quad
.
\endaligned
\end{equation}
Here $\{\cdot\}$ denotes a system (conjunction) of inequalities.
Analogously, each $y_{j}$ or $\lnot y_{j}$ can be coded using $\ex z$.
Note that the two strict inequalities in~\eqref{eq:bit} can be sharpened by multiplying both sides with $2^{i-1}$ to make all coefficients integer, and add $1$ to the RHS.

Now we show how to code~\eqref{eq:3SAT} using $\for \z$ with $\z \in \N^{3}$.
For each clause $(a_{k} \lor b_{k} \lor c_{k})$, we consider its negation $(\lnot a_{k} \land \lnot b_{k} \land \lnot c_{k})$.
Each term $\lnot a_{k}, \lnot b_{k}, \lnot c_{k}$ is still one of $x_{i},\lnot x_{i}, y_{i}, \lnot y_{i}$.
By~\eqref{eq:bit}, we have
\begin{equation*}%%\label{eq:pos}
(\lnot a_{k} \land \lnot b_{k} \land \lnot c_{k}) \quad \iff \quad \ex \z \in [0,2^{q})^3 : \Phi_{k}(x,y,\z),
\end{equation*}
where $\z \in \N^3$, and $\Phi_{k}$ is a conjunction of $6$ inequalities.
Taking negation, we have:
\begin{equation*}%%\label{eq:neg}
\aligned
( a_{k} \lor  b_{k} \lor  c_{k}) \quad &\iff \quad \for \z \in [0,2^{q})^3 : \lnot\Phi_{k}(x,y,\z),
\\
&\iff \quad \for \z \in [0,2^{q})^3 : \Psi_{k} (x,y,\z),
\endaligned
\end{equation*}
where $\Psi_{k}$ is a disjunction of $6$ inequalities.
Taking conjunction over all $k$ in~\eqref{eq:3SAT}, we have:
\begin{equation}\label{eq:F_Pr}
F(\wt x,\wt y) \quad \iff \quad \for \z \in [0,2^{q})^3 :  \Phi_{\ellR}(x,y,\z),
\end{equation}
where
\begin{equation}\label{eq:Phi_def}
\Phi_{\ellR}(x,y,\z)  \;=\;  %%x \in [0,2^{\ellR}) \land y \in [0,2^{p}) \land \,
\bigwedge_{k} \Psi_{k} (x,y,\z).
\end{equation}
Substituting~\eqref{eq:F_Pr} into \eqref{eq:3SATset}, we have~\eqref{eq:Ppo_Pr}.
If we assume in addition that $\L \in \poly$, then the circuit $C_{\ellR}$ can be built from a Turing Machine in time $\polyin_{\L}(\ellR)$, so the expression $\Phi_{\ellR}$ can also be found in time $\polyin_{\L}(\ellR)$.
This completes the proof.
\end{proof}

\begin{defi}\label{def:complement}
%%For a GF $f(\t) = \sum \t^{\x}$ in $n$ variables, the complement $\lnot f$ is the unique GF $h(\t)$ with $\supp(h) = \N^{n} \cpl \supp(f)$.
Given $f = \GFof(S; \t)$, where $S$ is a subset of a finite box $B \subset \N^{n}$.
The \emph{finite complement} $B \cpl f$ is $\GFof(B \cpl S; \t)$.
\end{defi}

\begin{defi}\label{def:union_intersection}
  Given $f_{1} = \GFof(S_{1};\t),\, \dots,\, f_{k} = \GFof(S_{k};\t)$ with $S_{1},\dots,S_{k} \subseteq \N^{n}$, the \emph{intersection} $f_{1} \cap \dots \cap f_{k}$ is $\GFof(S_{1} \cap\dots\cap S_{k}; \t)$.
%% \begin{equation*}
%% g = \GFof(S_{1} \cap\dots\cap S_{k};t).
%% \end{equation*}
  The \emph{union} $f_{1} \cup \dots \cup f_{k}$ is $\GFof(S_{1} \cup\dots\cup S_{k}; \t)$.
%% \begin{equation*}
%% \supp(h) = \supp(f_{1}) \cup \dots \cup \supp(f_{k}).
%% \end{equation*}
\end{defi}

%% \begin{defi}
%% For a language $\L \subseteq \{0,1\}^{\ellR}$, denote by $\GFof(\L_{\ellR};t)$ the GF
%% \begin{equation*}
%% \GFof(\L_{\ellR};t) \; \coloneqq \; \sum_{x \in \L} t^{x}.
%% \end{equation*}
%% %%This is abbreviated to $\GFof(\L_{\ellR})$ when the variable $t$ is clear from the context.
%% \end{defi}

\begin{theo}\label{th:Ppo_to_GF}
For every language $\L \in \Ppo$ and $\ellR > 0$, there exist a finite box $B_{\ellR}$ and short GF $f_{\ellR}(t,u,\v) \in \GF_{5,5}$ with $\supp(f_{\ellR}) \subseteq B_{\ellR}$, so that
\begin{equation}\label{eq:Ppo_GF}
  %%\sum_{x \in \L_{\ellR}} t^{x}
 \GFof(\L_{\ellR};t) \, =  \, \spec_{x}(B_{\ellR} \cpl \proj_{x,y}(f_{\ellR}) )
\end{equation}
and $\,\phi(B_{\ellR}),\, \phi(f_{\ellR}) \le \polyin_{\L}(\ellR)$.\footnote{Here $\phi(B{\ellR})$ denotes the total bit length of all sides in $B_{\ellR}$, written in binary.}
Furthermore, there exist polynomially many short GFs $\.p_{\ellR, 1},\.\dots,\,p_{\ellR, k_{\ellR}} \in \GF_{2,s}$ of finite supports, each of length $\polyin_{\L}(\ellR)$, so that:
\begin{equation}\label{eq:f_union}
\proj_{x,y}(f_{\ellR}) \; = \; p_{\ellR, 1} \cup \dots \cup p_{\ellR,k_{\ellR}}.
\end{equation}
Here $\GF_{2,s}$ is some fixed class that does not depend on $\L$.
If we assume in addition that $\L \in \poly$, then there is also an algorithm to compute $B_{\ellR}, f_{\ellR}$ and each $p_{\ellR,i}$ in time $\polyin_{\L}(\ellR)$.
\end{theo}

\begin{proof}
For the notations $\proj, \spec, \cup$ and $\cpl\ts$, we refer back to definitions~\ref{def:proj_and_spec},~\ref{def:complement} and~\ref{def:union_intersection}.
By the previous lemma, there is a $\Pr$ expression $\Phi_{\ellR}$ satisfying~\eqref{eq:Ppo_Pr}.
%%The expression $\Phi_{\ellR}$ was explicitly given in~\eqref{eq:Phi_def}.
First, define
\begin{equation*}
\aligned
B_{\ellR} &= \{(x,y) : x \in [0,2^{\ellR}),\, y \in [0,2^{p}) \}, \\
D_{\ellR} &= \{(x,y,\z) : x \in [0,2^{\ellR}),\, y \in [0,2^{p}),\, \z \in [0,2^{q})^3\},
\endaligned
\end{equation*}
where $\ellR,p$ and $q$ are from~\eqref{eq:Ppo_Pr}.
Define:
\begin{equation}\label{eq:Pr_to_GF}
f_{\ellR}(t,u,\v) \;=\; \sum_{\substack{(x,y,\z) \in D_{\ellR} \\ \lnot\Phi_{\ellR}(x,y,\z) \;}} t^{x} \; u^{y} \; \v^{\z}.
\end{equation}
Recall that $\Phi_{\ellR}$ is a quantifier free $\Pr$ expression with length $\textup{poly}_{\L}(\ellR)$.
%% So is the expression
%% \begin{equation*}
%% \lnot \Phi_{\ellR}(x,y,\z) \;\land\; (x,y,\z) \in B_{\ellR}.
%% \end{equation*}
Applying Theorem~\ref{th:W} to $\lnot \Phi_{\ellR}$, we can write $f_{\ellR}$  as a short GF in $\GF_{5,5}$  of finite support, which has length  $\phi(f_{\ellR}) \le \polyin(\phi(\Phi_{\ellR})) \le \polyin_{\L}(\ellR)$.
For the rest of the proof, we always assume $(x,y,\z) \in D_{\ellR}$.
We will simply write $\ex \z$ instead of $\ex \z \in [0,2^{q})^3$.
Projecting $f_{\ellR}$ on $(x,y)$, we have:
\begin{equation}\label{eq:ex_GF}
\proj_{x,y}(f_{\ellR}) \;=\; \sum_{(x,y) \,:\, \ex\z \, \lnot \Phi_{\ellR}(x,y,\z)} t^{x} u^{y}.
\end{equation}
Taking the complement of $\proj_{x,y}(f_{\ellR})$, which lies within the box $B_{\ellR}$, we have:
\begin{equation}\label{eq:sum_xy}
B_{\ellR} \cpl \proj_{x,y}(f_{\ellR})  \;=\; \sum_{(x,y) \,:\, \for \z \, \Phi_{\ellR}(x,y,\z)} t^{x} \; u^{y}.
\end{equation}
Recall that in the proof of Lemma~\ref{lem:Ppo_to_Pr}, the variable $y$ describes the values of non-input gates in the circuit $C_{\ellR}$, with input gates coming from $x$. Since the values of non-input gates are uniquely determined by the input gates, for every $x$ that satisfies $C_{\ellR}$ we have a unique $y$.
Substituting $u \gets 1$, the RHS in \eqref{eq:sum_xy} becomes $\GFof(\L_{\ellR};t)$.
%%$\sum_{x \in \L_{\ellR}} t^{x}$.
%% \begin{equation*}
%% \aligned
%% \spec_{x}(B_{\ellR} \cpl\proj_{x,y}(f_{\ellR})) \; = \; \sum_{x \,:\, \ex y \for \z \, \Phi_{\ellR}(x,y,\z)} t^{x} \;=\; \sum_{x \in \L_{\ellR}} t^{x}.
%% \endaligned
%% \end{equation}
We obtain~\eqref{eq:Ppo_GF}.

We proceed to show~\eqref{eq:f_union}.
Since $\lnot \Phi_{\ellR}$ is quantifier free with $5$ variables, we can apply Proposition~\ref{prop:W} on it and get:
\begin{equation*}
\lnot \Phi_{\ellR}(x,y,\z) \quad \iff \quad \bigvee_{i=1}^{k_{\ellR}} (x,y,\z) \in P_{\ellR,i} \cap \N^{5},
\end{equation*}
where $P_{\ellR,1},\.\dots,\,P_{\ellR,k_{\ellR}} \subseteq \R^{5}$ are disjoint polytopes (in the box $D_{\ellR}$) and $k_{\ellR} \le \polyin(\phi(\Phi_{\ellR})) \le \polyin_{\L}(\ellR)$.
Each polytope $P_{\ellR,i}$ also satisfies $\phi(P_{\ellR,i}) \le \polyin(\ellR)$.
Therefore:
\begin{equation}\label{eq:decomp_proj}
\ex \z  \; \lnot \Phi_{\ellR}(x,y,\z) \quad \iff \quad \bigvee_{i=1}^{k_{\ellR}} \ex \z \; \bigl[ (x,y,\z) \in P_{\ellR,i} \cap \N^{5} \bigr].
\end{equation}
Combined with~\eqref{eq:ex_GF}, we see that $(x,y) \in \supp(\proj_{x,y}(f_{\ellR}))$ if and only if it lies in the projection of some $P_{\ellR,i} \cap \N^{5}$.
By Theorem~\ref{th:BW}, for each $i$, we can find a short GF $p_{\ellR,i} \in \GF_{2,s}$ for the projection of $P_{\ellR,i} \cap \N^{5}$.
In other words, we have $p_{\ellR,i} \in \GF_{2,s}$ that satisfies:
\begin{equation*}
\supp(p_{\ellR,i}) = \{(x,y) : \ex \z \;  (x,y,\z) \in P_{\ellR,i} \cap \N^{5} \}.
\end{equation*}
Here $s$ is an absolute constant because each $P_{\ellR,i}$ has (fixed) dimension $5$.
We also have $\phi(p_{\ellR,i}) \le \polyin(\phi(P_{\ellR,i})) \le \polyin(\ellR)$.
The union of all short GFs $p_{\ellR,i}$ contains exactly all $(x,y)$ satisfying~\eqref{eq:decomp_proj}.
From~\eqref{eq:ex_GF} and~\eqref{eq:decomp_proj}, we have:
\begin{equation*}
\proj_{x,y}(f_{\ellR}) = p_{\ellR,1} \cup \dots \cup p_{\ellR,k_{\ellR}}.
\end{equation*}
This proves~\eqref{eq:f_union} and completes the proof.
\end{proof}

\begin{eg}\label{eg:squares}
Since $\textsc{SQUARES}$  and $\textsc{PRIMES}$ are both in $\poly$, we can represent all squares or primes up to $2^{\ellR}$ in the form~\eqref{eq:Ppo_GF}, with $f_{\ellR}$ and $B_{\ellR}$ computable in time $\polyin(\ellR)$.
%%\footnote{In fact, for $\textsc{PRIMES}$ this cane be done more economically by using the Miller--Rabin primality testing, which gives $\textsc{PRIMES}\in \BPP \subseteq \Ppo$ as in Theorem~\ref{th:Ppo_to_GF}}
\end{eg}

\begin{rem}
Even though \ts $\spec_{x}(B_{\ellR} \cpl \proj_{x,y}(f_{\ellR}))$ \ts may seem complicated, the specialization and complement are ``inexpensive operations'', which can be performed in polynomial time by theorems~\ref{th:BPGF1} and~\ref{th:BPGF2}. The main complexity resides in taking the projection of~$f$.

%% If the conditions of finite supports in Theorem~\ref{th:Ppo_to_GF} is relaxed, then we can take $B_{\ellR}$ to be the short GF of the whole $\N^{2}$, and $f_{\ellR}$ as in~\eqref{eq:Pr_to_GF} but without the condition $(x,y,\z) \in B_{\ellR}$.
%% In this case $B_{\ellR} \cpl f_{\ellR} = \lnot f_{\ellR}$ (see Definition~\ref{def:complement}).
%% This results in a simpler looking relation:
%% \begin{equation*}
%% \sum_{x\in\L_{\ellR}} t^{x} = \spec_{x}(\lnot\proj_{x,y}(f_{\ellR})).
%% \end{equation*}
\end{rem}

\begin{rem}\label{rem:UP}
  The same representation~\eqref{eq:Ppo_GF} applies to every language $\L$ in the complexity class $\unique\Ppo$. Such a language is characterized as follows.
  For every $\ellR$, there is a \emph{non-deterministic} polynomial-time Turing machine that accepts only $x \in \L_{\ellR}$, each with a unique accepting path.
  Given $\L \in \unique\Ppo$, we can obtain~\eqref{eq:Ppo_GF} by the same argument as above. In fact,~\eqref{eq:Ppo_GF} is an equivalent characterization of the class $\unique\Ppo$.
 Indeed, assume $\L_{\ellR}$ can be represented as~\eqref{eq:Ppo_GF}. Given $f_{\ellR}$, for any $x \in \L_{\ellR}$ there should be a unique certificate $y$ such that $(x,y) \in B_{\ellR} \cpl \proj_{x,y}(f_{\ellR})$, which is checkable in polynomial time by Proposition~\ref{prop:supp_proj_is_P}.
\end{rem}

\subsection{Compressing short GFs of finite supports}
We describe a technical tool which will be useful later.
This section can be skipped at first reading.

\begin{defi}\label{def:tau_map}
Consider $N = 2^{\ellR}$ and a vector $\x = (x_{1}, \dots, x_{d}) \in \N^{n}$ with $x_{i} \in [0,N)$ for all $1 \le i \le d$.
We define the $\tau_{N}$ map on $\x$ as:
\begin{equation*}
\tau_{N}(\x) \;=\; x_{1} + N x_{2} + \dots + N^{n-1} x_{d} \; \in [0,N^{n}).
\end{equation*}
For an array of vectors $\ov \x = (\x_{1},\dots,\x_{n})$ with $\x_{i} \in [0,N)^{n_{i}}$, we define:
\begin{equation*}
\tau_{N}(\ov \x) \;=\; (\tau_{N}(\x_{1}),\dots,\tau_{N}(\x_{k})) \; \in \; [0,N^{n_{1}}) \times \dots \times [0,N^{n_{k}}).
\end{equation*}
Finally, for a set $S \subseteq  [0,N)^{n_{1}} \times \dots \times [0,N)^{n_{k}}$, we define $\tau_{N}(S) = \{\tau_{N}(\ov \x) : \ov \x \in S\}$.
\end{defi}

The following technical tool allows us to reduce the number of variables in a short GF of finite support.

\begin{lem}\label{lem:compress}
Fix $k,s$ and $n_{1},\dots,n_{k} \in \N$.
Let $n = n_{1} + \ldots + n_{k}$.
\smallskip

\nin\textup{a) \textbf{Compressing}:} Given a short GF $g(\ov \t) = \sum \t_{1}^{\x_{1}} \dots \t_{k}^{\x_{k}}$ of finite support in the class $\GF_{n,s}\.$, there exist an $N = 2^{\ellR}$ with $\,\supp(g) \subseteq  [0,N)^{n_{1}} \times \dots \times [0,N)^{n_{k}}$ and a short GF $f(\u) = \sum u_{1}^{z_{1}} \dots u_{k}^{z_{k}}$ in the class $\GF_{k,s}$ so that
\begin{equation}\label{eq:compress}
\supp(f) = \tau_{N}(\supp(g)) \;\subseteq\; [0,N^{n_{1}}) \times \dots \times [0,N^{n_{k}}).
\end{equation}
Both $f$ and $N$ can be computed in time $\polyin(\phi(g))$ with $\phi(f), \log N \le \polyin(\phi(g))$.
\smallskip

\nin\textup{b) \textbf{Decompressing}:} Conversely, given $f(\u) = \sum u_{1}^{z_{1}} \dots u_{k}^{z_{k}} \in \GF_{k,s}$ and $N = 2^{\ellR}$ such that
\begin{equation*}
\supp(f) \;\subseteq\; [0,N^{n_{1}}) \times \dots \times [0,N^{n_{k}}),
\end{equation*}
there exists $g(\ov \t) = \sum \t_{1}^{\x_{1}} \dots \t_{k}^{\x_{k}} \in \GF_{n, n+s}$ with $\supp(g) \subseteq [0,N)^{n_{1}} \times \dots \times [0,N)^{n_{k}}$ which satisfies~\eqref{eq:compress}.
The short GF $g$ can be computed in time $\polyin(\phi(f)+\log N)$.
\end{lem}

Proof for the lemma is technical and is postponed until Section~\ref{s:lemma-proof}.
We note that the compression map $\tau_{N}$ in Definition~\ref{def:tau_map} is similar
to that used in the polynomial identity testing algorithm of Klivans and Spielman~\cite{KS}.
Using Lemma~\ref{lem:compress}, we can reduce the number of variables of $f_{\ellR}$
in~\eqref{eq:Ppo_GF} down to~$3$.

\begin{cor}\label{cor:Ppo_to_GF_3}
For every language $\L \in \Ppo$ and $\ellR > 0$, there exist a finite box $B_{\ellR}$ and short GF $f_{\ellR}(t,u,v) \in \GF_{3,5}$ with $\supp(f_{\ellR}) \subseteq B_{\ellR}$, so that~\eqref{eq:Ppo_GF} holds.
%% \begin{equation*}%%\label{eq:Ppo_GF}
%% \sum_{x \in \L_{\ellR}} t^{x} \, =  \, \spec_{x}(B_{\ellR} \cpl \proj_{x,y}(f_{\ellR}) )
%% \end{equation*}
%%and $\,\phi(B_{\ellR}), \phi(f_{\ellR}) = \polyin_{\L}(\ellR)$.
The rest is identical to Theorem~\ref{th:Ppo_to_GF}.
\end{cor}

\begin{proof}
We have~\eqref{eq:Ppo_GF} with $f_{\ellR}(t,u,\v) = \sum t^{x} u^{y} \v^{\z} \in \GF_{5,5}$ a short GF of finite support in five variables $(t,u,v_{1},v_{2},v_{3})$.
Using part a) of Lemma~\ref{lem:compress}, we can compress $\z$ into  a single-variable $w$, leaving both $x$ and $y$ unchanged.
In other words, $t^{x} u^{y} \v^{\z}$ becomes $t^{x} u^{y} v^{w}$.
Note that $\proj_{x,y}$ is not affected by compression.
This gives us a short GF $\wt f_{\ellR} \in \GF_{3,5}$ with
\begin{equation*}
\proj_{x,y}(\wt f_{\ellR}) \; = \; \proj_{x,y}(f_{\ellR}) \quad \text{and} \quad \phi(\wt f_{\ellR}) \le \polyin(\phi(f_{\ellR})) \le \polyin(\ellR).
\end{equation*}
So we can substitute $\wt f_{\ellR}$ for $f_{\ellR}$ in~\eqref{eq:Ppo_GF}.
\end{proof}

\bigskip

\section{Short GFs and the non-uniform polynomial hierarchy}  \label{sec:nu-hier}

The non-uniform polynomial hierarchy $\PH\po$ starts with $\Ppo = \SigmaP_{0}\po = \PiP_{0}\po$ at the $0$th level.
%% The higher levels defined as:
%% \begin{equation*}
%% \SigmaP_{k+1}\po = \NP^{\SigmaP_{k}\po}  \quad \text{and} \quad \PiP_{k+1}\po = \coNP^{\SigmaP_{k}\po}.
%% \end{equation*}
For $k>0$, a language $\L$ is in $\SigmaP_{k}\po$ if for every $\ellR > 0$,
there is a circuit $C_{\ellR}$ of size $\polyin_{\L}(\ellR)$ so that for every
string $\wt x$ of length $\ellR$ we have:
\begin{equation*} 
\wt x \in \L_{\ellR} \quad \iff \quad \ex \wt y_{1} \; \for \wt y_{2} \dots \; Q_{k} \wt y_{k} : C_{\ellR}(x,y_{1},\dots,y_{k}) = 1.
\end{equation*}
Here $Q_{1},\dots,Q_{k}$ are $k$ alternating quantifiers with $Q_{1}=\ex$, and $\wt y_{1},\dots, \wt y_{k}$ are binary strings of length polynomial in~$\ellR$.
For $\PiP_{k}\po$ the alternating quantifiers are reversed ($Q_{1}=\for$).
We have a the following analogue to Lemma~\ref{lem:Ppo_to_Pr} for each level in $\PH\po$:

\begin{lem}\label{lem:PH_to_Pr}
For every language $\L \in \SigmaP_{k}\po$ and $\ellR > 0$, there exists a quantifier free $\Pr$ expression in $k+4$ variables $x \in \N$, $\y \in \N^{k}$, $\z \in \N^{3}$, so that $\wt x \in \L_{\ellR}$ if and only if:
\begin{equation}\label{eq:PH_Pr}
x \in [0,2^{\ellR}) \,\land\, \Big[ Q_{1} y_{1} \in [0,2^{p_{1}}) \, \dots \, Q_{k} y_{k} \in [0,2^{p_{k}}) \; Q_{k+1} \z \in [0,2^{q})^3 : \Phi_{\ellR}(x,\y,\z) \Big].
\end{equation}
Here $Q_{1}, \dots, Q_{k+1}$ are $k+1$ alternating quantifiers with $Q_{1} = \ex$.
Moreover, we have $p_{1},\dots,p_{k},q,\phi(\Phi_{\ellR}) \le \polyin_{\L}(\ellR)$.
For the case $\L \in \PiP_{k}\po$, the quantifiers $Q_{i}$ are reversed.
\end{lem}

\begin{proof}
For simplicity, we prove the claim for $\L \in \SigmaP_{1} = \NP\po$.
The higher levels $\SigmaP_{k}\po$ and $\PiP_{k}\po$ can be argued similarly.
Since $\L \in \NP\po$, for each $\ellR$, there is a circuit $C_{\ellR}$ of size $\polyin_{\L}(\ellR)$ such that
\begin{equation}\label{eq:NP_cert}
\wt x \in \L_{\ellR} \quad \iff \quad \ex\, \wt c \in \{0,1\}^{\rCr} \. : \. C_{\ellR}(\wt x,\wt c) = 1\ts ,
\end{equation}
where $\rCr \le \textup{poly}_{\L}(\ellR)$ is the certificate length.
%%The condition $M(\wt x,\wt c) = 1$ can be converted to $C_{\ellR}(\wt x,\wt c) = 1$, where $C_{\ellR}$ is a Boolean circuit of size $\textup{poly}_{\L}(\ellR)$ having both $\wt x$ and $\wt c$ as inputs.
The circuit $C_{\ellR}$ also has $p$ non-input gates with $p \le \polyin_{\L}(\ellR)$.
Let $p' = \rCr + p$.
Note that the certificate gates $\wt c \in \{0,1\}^{\rCr}$ and the non-input gates 
$\wt y \in \{0,1\}^{p}$ can be coded by a single integer $y \in [0,2^{p'})$.
The argument now proceeds similarly to Lemma~\ref{eq:Ppo_Pr} with $p'$ in place 
of~$p$.  \end{proof}

\begin{rem}
In~\cite[Lemma~5.2]{G}, Gr\"{a}del gave a similar representation to~\eqref{eq:PH_Pr}.
In his representation, each string $\wt x = (x_{1},\dots,x_{\ellR}) \in \{0,1\}^{\ellR}$ is not simply mapped to its binary integer value, but to:
\[
x = p_{1}^{x_{1}} \dots p_{\ellR}^{x_{\ellR}} \, q_{1}^{1 - x_{1}} \dots q_{\ellR}^{1 - x_{\ellR}},
\]
where $p_{1},\dots,p_{\ellR}, q_{1},\dots, q_{\ellR}$ are the first $2\ellR$ prime numbers.
\end{rem}

\begin{rem}
From this result, we see that the problem of deciding $\Pr$ sentences of the form \ts $\ex\y\ts \for\z \, \Phi(\y,\z)\ts $ is at least $\NP$-hard.
Sch\"{o}ning \cite{S} showed that the problem is $\NP$-complete even for the case $\ts \ex y \ts \for z \, \Phi(y,z)$, i.e., when both variables are singletons.
\end{rem}

\begin{defi}\label{def:anti-projection}
Let $f = \sum \t^{\x} \u^{\y} = \GFof(S; \t,\u)$, where $S$ is a subset of a finite box $I \times J$.
%%For a GF $f(\t,\u) = \sum \t^{\x} \u^{\y}$ with $\supp(f) \subseteq I \times J$,
The \emph{anti-projection} $\cproj_{\x}(f)$ is $F(I; \t) - \proj_{\x}(f)$,
%% \begin{equation*}
%% \cproj_{\x}(f) = I \cpl \proj_{\x}(f),
%% \end{equation*}
where the projection $\proj_{\x}(f)$ is from Definition~\ref{def:proj_and_spec}.
%%and the complement $\cpl$ is from Definition~\ref{def:complement}.
The box $I \times J$ is always specified before taking the anti-projection.
\end{defi}

\begin{theo}\label{th:PH_to_GF}
For every language $\L \in \SigmaP_{k}\po$ and $\ellR > 0$, there exists a short GF $f_{\ellR} \in \GF_{k+2,k+4}$ of the form $f_{\ellR}(t,u_{1},\dots,u_{k},v) = \sum t^{x} u_{1}^{y_{1}} \dots u_{k}^{y_{k}} v^{z}$ such that
\begin{equation}\label{eq:PH_GF}
 %% \sum_{x \in \L_{\ellR}} t^{x}
\GFof(\L_{\ellR};t)  \, =  \proj_{x}\Big( \cproj_{x,y_{1}} \big( \proj_{x,y_{1},y_{2}} ( \cdots (f_{\ellR}) \cdots ) \big) \Big),
\end{equation}
where the $k$ alternating projections and anti-projections are taken in a finite box
\begin{equation*}
B_{\ellR} = [0,2^{\ellR}) \times [0,2^{p_{1}}) \times \dots \times [0,2^{p_{k}}) \times [0,2^{q}).
\end{equation*}
Moreover, we have $p_{1},\dots,p_{k},q,\phi(f_{\ellR}) \le \polyin_{\L}(\ellR)$.
%% and there is also an algorithm to compute $B$ and $f_{\ellR}$ in time $\polyin_{\L}(\ellR)$.
For $\L \in \PiP_{k}\po$, the projections and anti-projections are reversed.
\end{theo}

\begin{proof}
By Lemma~\ref{lem:PH_to_Pr}, we can represent $\L_{\ellR}$ in the form~\eqref{eq:PH_Pr}.
Applying the same argument in Theorem~\ref{th:Ppo_to_GF}, we get $f_{\ellR}(t,u_{1},\dots,u_{k},\v) = \sum t^{x} u_{1}^{y_{1}} \dots u_{k}^{y_{k}} \v^{\z} \in \GF_{k+4,k+4}$ that satisfy~\eqref{eq:PH_GF}.
Applying Lemma~\ref{lem:compress} a), we can compress the last three variables $\v^{\z} = v_{1}^{z_{1}} v_{2}^{z_{2}} v_{3}^{z_{3}}$ into just one variable $v^{w}$ without affecting the projections (see the proof of Corollary~\ref{cor:Ppo_to_GF_3}).
This reduces $f_{\ellR}$ to a short GF in $\GF_{k+2,k+4}$.
\end{proof}

\begin{rem}\label{rem:in_PH}
If in addition $L \in \PH$, then both $\Phi_{\ellR}$ and $f_{\ellR}$ in Lemma~\ref{lem:PH_to_Pr} and Theorem~\ref{th:PH_to_GF} can be computed in time $\polyin_{\L}(\ellR)$.
Indeed, if $\L \in \PH$, the circuit $C_{\ellR}$ for $\L_{\ellR}$ in Lemma~\ref{lem:PH_to_Pr}'s proof can be automatically generated by some polynomial time Turing Machine $M$.
We can convert $C_{\ellR}$ to $\Phi_{\ellR}$ in polynomial time, which allows us to find $f_{\ellR}$.
\end{rem}

As a consequence, we obtain the following result.

\begin{cor}\label{cor:2_proj_hard}
Assume we are given $a_{0} \in \N$, a short GF $f(t,u,v) = \sum t^{x} u^{y} v^{z} \in \GF_{3,5}$, and a finite box $B \subset \N^{3}$ with $\supp(f) \subseteq B$.
Then deciding whether $a_{0} \in \supp(h)$ is $\NP$-complete,
where $h = \proj_{x}(\cproj_{x,y}(f))$.
Here the projection and anti-projection are taken within $B$.
\end{cor}

\begin{proof}
If $a_{0} \in \supp(h)$, there exists some $b_{0}$ so that $(a_{0},b_{0})$ lies in the support of $\cproj_{x,y}(f)$.
Since $\cproj_{x,y}(f)$ is taken within $B$, which is bounded, both $a_{0}$ and $b_{0}$ must have polynomial lengths.
Given such a certificate $b_{0}$, we can verify if $(a_{0},b_{0})$ lies in the support of $\proj_{x,y}(f)$ in polynomial time, by applying Proposition~\ref{prop:supp_proj_is_P}.
Taking a negation, we can also check whether $(a_{0},b_{0})$ lies in the anti-projection $ \cproj_{x,y}(f)$.
This shows the problem is in $\NP$.

The problem is also $\NP$-hard.
Indeed, let $\L$ be an $\NP$ language.
Applying Theorem~\ref{th:PH_to_GF} for the case $\L \in \NP$, we have $\GFof(\L_{\ellR}; t) \. = \. \proj_{x}\bigl(\cproj_{x,y}(f_{\ellR})\bigr),$
%% \begin{equation*}
%% \sum_{x \in \L_{\ellR}} \ts t^{x} \. = \. \proj_{x}\bigl(\cproj_{x,y}(f_{\ellR})\bigr),
%% \end{equation*}
where $f_{\ellR}$ is supported inside a box $B_{\ellR}$.
By Remark~\ref{rem:in_PH}, we can compute $f_{\ellR}$ and $B_{\ellR}$ in polynomial time.
So checking $x \in \L_{\ellR}$ is equivalent to checking $x \in \supp(h_{\ellR})$, where $h_{\ellR} = \proj_{x}(B_{\ellR} \cpl \proj_{x,y}(f_{\ellR}))$.
\end{proof}

\begin{rem}
Compared to Proposition~\ref{prop:supp_proj_is_P}, we see that it is no longer easy to check for membership after taking two separate projections on a short GF.
\end{rem}

\bigskip

\section{A hierarchy of generating functions}\label{sec:GF}

We introduce a hierarchy $\GH$ of languages expressible as projections of generating functions.
First, we define the lowest level $\Gzero = \SigmaG_{0} = \PiG_{0}$.

\begin{defi}\label{def:Gzero}
For a language $\L \in \{0,1\}^{*}$, we say that $\L \in \Gzero$ if there is an $s > 0$ so that for every $\ellR > 0$, we can represent $\GFof(\L_{\ellR}; t) = f_{\ellR}(t)$
%% \begin{equation*}
%%   \sum_{x \in \L_{\ellR}} t^{x} = f_{\ellR}(t),
%% \end{equation*}
where $f_{\ellR} \in \GF_{1,s}$ and $\phi(f_{\ellR}) \le \polyin_{\L}(\ellR)$.
In other words, every segment $\L_{\ellR}$ can be represented as a
short GF of polynomial length in some fixed class $\GF_{1,s}$.
\end{defi}

We define higher classes $\SigmaG_{k}$ and $\PiG_{k}$ by taking repeated projections/anti-projections.

\begin{defi}\label{def:SigmaG}
For a language $\L \in \{0,1\}^{*}$, we say that $\L \in \SigmaG_{k}$ if there is an $s > 0$ so that for every $\ellR > 0$, we can represent:
\begin{equation}\label{eq:SigmaG}
%%  \sum_{x \in \L_{\ellR}} t^{x}
  \GFof(\L_{\ellR}; t) \, = \; \proj_{x}\Big( \cproj_{x,y_{1}} \big( \proj_{x,y_{1},y_{2}} ( \cdots (f_{\ellR}) \cdots ) \big) \Big),
\end{equation}
where $f_{\ellR}(t,u_{1},\dots,u_{k}) = \sum t^{x} u_{1}^{y_{1}} \dots u_{k}^{y_{k}} \in \GF_{k+1, s}$ is supported inside a finite box $B_{\ellR}$, with both $\phi(B_{\ellR}),\phi(f_{\ellR}) \le \polyin_{\L}(\ellR)$.
The $k$ alternating projections/anti-projections are taken within $B_{\ellR}$.
The class $\PiG_{k}$ is defined similarly, with the projections/anti-projections in~\eqref{eq:SigmaG} reversed.
Alternatively, $\L \in \PiG_{k}$ if and only if the complement language $\lnot \L$ is in $\SigmaG_{k}$.
\end{defi}

\begin{defi}\label{def:GH}
$\GH$ is the union of all $\SigmaG_{k}$ and $\PiG_{k}$ for all $k \ge 0$.
\end{defi}

We list some properties of $\GH$:
\smallskip
\begin{itemize}
\setlength\itemsep{1em}
\item $\SigmaG_{k},\, \PiG_{k} \,\subseteq\, \SigmaG_{k+1} \cap \PiG_{k+1}$ for all $k \ge 0$.

\item $\Gzero,\, \SigmaG_{1},\, \PiG_{1} \,\subseteq\, \Ppo$ (propositions~\ref{prop:supp_is_P} and~\ref{prop:supp_proj_is_P}).

\item $\Ppo \subseteq \unique\PiG_{1}$, the subclass of $\SigmaG_{2}$ with only $\spec_{x}$ and $\cproj_{x,y}$ (Theorem~\ref{th:Ppo_to_GF}).

\item In fact, $\unique\PiG_{1} = \unique\Ppo$ (Remark~\ref{rem:UP}).

\item $\SigmaP_{k}\po \subseteq \SigmaG_{k+1},\, \PiP_{k}\po \subseteq \PiG_{k+1}$ for all $k \ge 1$ (Theorem~\ref{th:PH_to_GF}).
\end{itemize}
\smallskip

The last property can actually be strengthened to:

\begin{theo}\label{th:GH_PHpo}
$\SigmaP_{k}\po = \SigmaG_{k+1} $ and $\,\PiP_{k}\po = \PiG_{k+1}$ for every $k \ge 1$. So $\GH = \PH\po$, i.e., $\GH$ is exactly the non-uniform version of $\PH$.
\end{theo}
\begin{proof}
Theorem~\ref{th:PH_to_GF} already showed inclusion in one direction.
For the other direction, assume $\L \in \SigmaG_{k+1}$. From Definition~\ref{def:SigmaG}, for every $\ellR > 0$, we have:
\begin{equation*}
%%  \sum_{x \in \L_{\ellR}} t^{x}
  \GFof(\L_{\ellR}; t) \, = \, \proj_{x}\Big( \cproj_{x,y_{1}} \big( \proj_{x,y_{1},y_{2}} ( \cdots (f_{\ellR}) \cdots ) \big) \Big),
\end{equation*}
where $f_{\ellR}$ is a short GF of length $\polyin_{\L}(\ellR)$ in some fixed class $\GF_{k+2,s}$.
Here we are taking $k+1$ alternating projections and anti-projections on $f_{\ellR}(x,y_{1},\dots,y_{k+1}) = \sum t^{x} u_{1}^{y_{1}} \dots u_{k+1}^{y_{k+1}}$ within some finite box $B_{\ellR}$.
Note that by Proposition~\ref{prop:supp_proj_is_P}, we can check in polynomial time if $(x,y_{1},\dots,y_{k})$ lies in the inner most projection/anti-projection.
So given $f_{\ellR}$ as an advice string, we can decide if $x \in \L_{\ellR}$ by calling a $\SigmaP_{k}$ oracle for the remaining $k$ projections/anti-projections.
This implies $\L \in \SigmaP_{k}\po$.
The case $\L \in \PiG_{k+1}$ is similar.
\end{proof}

\bigskip

\section{Short GFs have long projections} \label{sec:long-proj}

\subsection{Proof of Theorem~\ref{th:main_1}}

\begin{theo}\label{th:weaker_assumption}
If $\sharpP \not\subseteq \FPpo$, then $\Gzero \subsetneq \Ppo$.
\end{theo}

\begin{proof}
We saw in Section~\ref{sec:GF} that $\Gzero \subseteq \Ppo$.
Now we show $\Ppo$ is strictly larger than~$\Gzero$.
Let $\#\L$ be an $\sharpP$-complete problem (e.g. $\textsc{\#3SAT}$), which is outside of $\FPpo$ by the assumption $\sharpP \not\subseteq \FPpo$.
Associated to $\#\L$ is a polynomial time Turing machine $M$.
Given $\wt x \in \{0,1\}^{\ellR}$, $\#\L$ asks for the number of certificates $\wt c \in \{0,1\}^{\ellR}$ that satisfy $M(\wt x, \wt c) = 1$.
Define a language:
\begin{equation}\label{eq:M_lang}
\M = \{(\wt x, \wt c) : \text{length}(\wt x) = \text{length}(\wt c) \; \text{ and } \; M(\wt x,\wt c) = 1\}.
\footnote{In general, the instance $\wt x$ and certificate $\wt c$ can have different lengths. However, the Turing Machine $M$ can always be modified to accept only $\wt c$ and $\wt x$ of equal lengths.}
\end{equation}
Since $M$ runs in polynomial time, we also have $\M \in \Ppo$.
We show that $\M \notin \Gzero$.

Assume the contrary, i.e., $\M \in \Gzero$. Then there is a fixed $s$ so that for every $\ellR > 0$, we have $\M_{\ellR} = \supp(f_{\ellR})$, where $f_{\ellR} \in \GF_{1,s}$ and $\phi(f_{\ellR}) \le \polyin(\ellR)$.
Let $x,c \in [0,2^{\ellR})$ be the integers corresponding to $\wt x, \wt c \in \{0,1\}^{\ellR}$.
Then the concatenated string $(\wt x, \wt c)$ corresponds to $x + 2^{\ellR} c$.
We assumed that there is an $f_{2\ellR} \in \GF_{1,s}$ such that
\begin{equation*}
\phi(f_{2\ellR}) \le \polyin(\ellR) \quad \text{and} \quad \sum_{(\wt x, \wt c) \in \M_{2\ellR}} t^{x+2^{\ellR}c} = f_{2\ellR}(t).
\end{equation*}
Given $\wt x \in \{0,1\}^{\ellR}$, we must compute the number of $\wt c \in \{0,1\}^{\ellR}$ which satisfy $(\wt x, \wt c) \in \M_{2\ellR}$. Define
\begin{equation}\label{eq:sum_c}
g_{x}(t) = \sum_{0 \le c < 2^{\ellR}} t^{x + 2^{\ellR}c} = t^{x} \frac{1-t^{2^{2\ellR}}}{1-t^{2^{\ellR}}}.
\end{equation}
We have $\phi(g_{x}) \le \polyin(\ellR)$.
We also have $f_{2\ellR} \in \GF_{1,s}$ and $g_{x} \in \GF_{1,1}$.
Therefore, by Theorem~\ref{th:BPGF2}, the short GF $h_{x} = f_{2\ellR} \.\star\. g_{x}$ can be computed in time $\polyin(\phi(f_{2\ellR}) + \phi(g_{x})) \le \polyin(\ellR)$.
The number of certificates $\wt c$ for $\wt x$ is simply $h_{x}(1)$.
This substitution can be computed in time $\polyin(\ellR)$ by Theorem~\ref{th:BPGF1}.

To summarize, the short GF $f_{2\ellR}$ gives us a polynomial size circuit to solve $\#\L$ for all inputs $\wt x \in \{0,1\}^{\ellR}$ in time $\polyin(\ellR)$.
We conclude that $\#\L \in \FPpo$, a contradiction.
\end{proof}

Now we can formulate Theorem~\ref{th:main_1} in precise terms:

\begin{cor}\label{cor:proj_long_strong}
If $\sharpP \not\subseteq \FPpo$, then $\GH$ does not collapse to its $0$th level $\Gzero$.
In other words, there is a sequence $\big\{f_{\ellR}\big\}_{\ellR > 0}$ in some fixed class $\GF_{2,s}$ with $\phi(f_{\ellR}) \le \polyin(\ellR)$ so that for every $d$, $\proj_{x}(f_{\ellR})$ cannot be written as a short GF $h_{\ellR} \in \GF_{1,d}$ with $\phi(h_{\ellR}) \le \polyin(\ellR)$.
\end{cor}

\begin{proof}
Recall that $\Gzero \subseteq \Ppo \subseteq \GH$ (Section~\ref{sec:GF}).
Now this follows from Theorem~\ref{th:weaker_assumption}.
\end{proof}

\subsection{A partial converse}

One can ask if the above argument in the proof above can be reversed, i.e.,
if $\sharpP \subseteq \FPpo$, does it imply that $\GH$ collapses to $\Gzero$?
We present below a weaker result.

%%\begin{defi}
Recall from Section~\ref{sec:GF} that $\unique\PiG_{1}$ the subclass of $\SigmaG_{2}$ that uses only $\spec_{x}$ and $\cproj_{x,y}$.
In other words, $\L \in \unique\PiG_{1}$ if for every $\ellR > 0$, we have $\GFof(\L_{\ellR};t) = \spec_{x}(\cproj_{x,y}(f_{\ellR}))$
%% \begin{equation*}
%% \sum_{x\in\L_{\ellR}} t^{x} = \spec_{x}(\cproj_{x,y}(f_{\ellR}))
%% \end{equation*}
for some $f_{\ellR}$ in some fixed class $\GF_{3,s}$ with $\phi(f_{\ellR}) \le \polyin_{\L}(\ellR)$.
We also know that $\unique\PiG_{1} = \unique\Ppo$.
%%\end{defi}

\begin{prop}\label{prop:partial_converse}
If $\sharpP \subseteq \FPpo$, then $\GH$ collapses to $\unique\PiG_{1}$.
%%If $\,\GH$ does not collapse to $\unique\PiG_{1}$, then $\sharpP \not\subseteq \FPpo$.
\end{prop}

%%In other words, if $\ts\unique\PiG_{1} \subsetneq\GH$, then $\sharpP \not\subseteq \FPpo$.

\begin{proof}
%%Recall that $\unique\PiG_{1}$ is the subclass of $\SigmaG_{2}$ which uses only $\spec_{x}$ and $\cproj_{x,y}$.
Since $\GH = \PH\po$ and $\unique\PiG_{1} = \unique\Ppo$, it equivalent to show $\PH\po = \unique\Ppo$.
In fact, we have a stronger collapse, namely $\PH\po = \Ppo$.
This follows easily from \emph{Toda's theorem} (see e.g.~\cite[Sec.\ 9.3]{AB}).
%% Assume $\sharpP \subseteq \FPpo$, i.e., every counting problem has polynomial size circuits.
Indeed, by Toda's theorem , we have $\PH \subseteq \poly^{\textsc{\#SAT}}$.
Replacing the $\textsc{\#SAT}$ oracle by polynomial size circuits, we have $\PH \subseteq \poly^{\Ppo} = \Ppo$.
Taking the non-uniform version of $\PH$, we still have $\PH\po \subseteq \Ppo$.
%% We showed in Theorem~\ref{th:GH_PHpo} that $\PH\po = \GH$.
%% So $\GH$ collapses to $\Ppo \subseteq \unique\Ppo = \unique\PiG_{1}$.
%% From Corollary~\ref{cor:Ppo_to_GF_3}, we also have $\Ppo \subseteq \unique\PiG_{1}$.
%% Together, they imply $\GH = \PH\po  \subseteq \Ppo \subseteq \unique\PiG_{1}$.
\end{proof}

\begin{rem}\label{rem:hard}
The proposition implies that proving $\GH$ does not collapse to between its 1st
and 2nd levels is at least as hard as showing $\sharpP \not\subseteq \FPpo$.
However, there might still be hope of showing that $\textsf{GH}$ does not
collapse to its $0$th level~$\Gzero$, e.g., by proving
Conjecture~\ref{conj:squares-intro}.
%
% exhibiting some short GFs with no polynomial length projections.
\end{rem}

\begin{rem}\label{rem:strong_collapse}
We do not claim that Proposition~\ref{prop:partial_converse}
is a new collapse result assuming $\sharpP \subseteq \FP\po$.
Here we are only putting things in the context of short GFs.
Observe that $\sharpP \subseteq \FP\po$ implies $\NP \subseteq \Ppo$.
In turn, $\NP \subseteq \Ppo$ implies $\PH = \textsf{S}_{2}^{\poly}$ (see~\cite{C}),
which is the strongest collapse currently known, assuming $\NP \subseteq \Ppo$.
Note that the classical \emph{Karp--Lipton theorem} (see e.g.~\cite{AB,MM,Pap}), says
that $\NP \subseteq \Ppo$ implies $\PH = \SigmaP_{2}$,
which is weaker because $\textsf{S}_{2}^{\poly} \subseteq \SigmaP_{2} \cap \PiP_{2}$.
\end{rem}

\bigskip

\section{Intersections, unions and Minkowski sums of short GFs}  \label{sec:int-unions}

\subsection{Proof of Theorem~\ref{th:main_2}}

Below is the precise statement of Theorem~\ref{th:main_2}.

\begin{theo}\label{th:union_long}
Assume $\sharpP \not\subseteq \FPpo$.
Then there is an $s > 0$ and a family of finite subsets $\big\{ S_{\ellR} \big\}_{\ellR > 0}$ with each $S_{\ellR} = \{p_{\ellR,1}, \dots, p_{\ellR, k_{\ellR}}\} \subset \GF_{1,s}$ so than the following hold:
\noindent
\begin{itemize}[leftmargin = 2em]
\item[\textup{a)}]
The total length of all $\,p_{\ellR,i}$ in $S_{\ellR}$ is $\polyin(\ellR)$.

\item[\textup{b)}] For every fixed $d$, the intersection/union of all $\,p_{\ellR,i}$ in $S_{\ellR}$ cannot be written as a short GF $h_{\ellR} \in \GF_{2,d}$ with $\phi(h_{\ellR}) \le \polyin(\ellR)$.
\end{itemize}
\end{theo}

\begin{proof}
By Theorem~\ref{th:weaker_assumption}, there exists a language $\L \in \Ppo$ which is outside of $\Gzero$.
By Theorem~\ref{th:Ppo_to_GF},
%Corollary~\ref{cor:Ppo_to_GF_3},
for every $\ellR > 0$, we can represent:
\begin{equation*}
%%  \sum_{x \in \L_{\ellR}} t^{x}
\GFof(\L_{\ellR}; t) \, =  \, \spec_{x}(B_{\ellR} \cpl \proj_{x,y}(f_{\ellR}) ) \quad \text{and} \quad \proj_{x,y}(f_{\ellR}) \; = \; p_{\ellR, 1} \cup \dots \cup p_{\ellR,k_{\ellR}},
\end{equation*}
where $f_{\ellR} \in \GF_{5,5},\; p_{\ellR,i} \in \GF_{2,s}$ and $\phi(B_{\ellR}),\, \phi(f_{\ellR}),\, \sum \phi(p_{\ellR,i}) \le \polyin(\ellR)$.
Here $s$ is some universal constant.

Let $S_{\ellR} = \{p_{\ellR,1},\dots,p_{\ellR,k_{\ellR}}\}$.
This family $\big\{S_{\ellR}\big\}$ satisfies condition~a).
We show that the union of $p_{\ellR,i}$ cannot be written as a short GF of length $\polyin(\ellR)$.
Indeed, assume there is $d$ for which we can write $\proj_{x,y}(f_{\ellR}) = p_{\ellR, 1} \cup \dots \cup p_{\ellR,k_{\ellR}}\,$ as $h_{\ellR} \in \GF_{2,d}$ with $\phi(h_{\ellR}) \le \polyin(\ellR)$.
By Theorem~\ref{th:BPGF2}, the complement $B_{\ellR} \cpl h_{\ellR}$ can be written as a short GF  $g_{\ellR} \in \GF_{2,2d}$ of length $\polyin(\ellR)$.
Taking the specialization $\spec_{x}(g_{\ellR})$, we still have a short GF in $\GF_{2,2d}$ of length $\polyin(\ellR)$, which represents $\L_{\ellR}$.
Since this holds for all $\ellR > 0$, we have $\L \in \Gzero$, a contradiction.
So the family $\big\{S_{\ellR}\big\}$ also satisfies b).

Note that each $p_{\ellR,i}$ still has $2$ variables $x,y$.
By Lemma~\ref{lem:compress} part a), we can compress each $p_{\ellR,i}$ into a single variable short GF $\,\wt p_{\ellR,i} \in \GF_{1,s}$ of polynomial length.
Then the new subsets $\wt S_{\ellR} = \{ \wt p_{\ellR,1}, \dots, \wt p_{\ellR,k_{\ellR}}\} \subset \GF_{1,s}$ still satisfy condition a).
We show they still satisfy condition~b).
Indeed, note that compressing/decompression preserves intersection and union.
So if $\,\wt p_{\ellR,i}\,$ has a polynomial length union then Lemma~\ref{lem:compress} part b) allows us the decompress it into a polynomial length union of $p_{\ellR,i}$.
This completes the proof for the case of union.
The case of intersection follows by taking complements of $p_{\ellR,i}$.
\end{proof}

\subsection{Proof of Theorem \ref{th:main_3}}
%%We prove Theorem~\ref{th:main_3} under the weaker assumption in Conjecture~\ref{conj:some_not_short}.

\begin{defi}
  Given two GFs $a = \GFof(S_{1}; \t)$ and $b = \GFof(S_{2}; \t)$ with $S_{1},S_{2} \subseteq \N^{n}$,  the \emph{Minkowski sum} $a \oplus b$ is $\GFof(S_{1}\oplus S_{2};\t)$,
  %$a(\t)$ and $b(\t)$ is the unique GF $c(\t)$ that satisfies
%% \begin{equation*}
%% \supp(c) = \supp(a) \oplus \supp(b),
%% \end{equation*}
where $S_{1} \oplus S_{2}$ is the usual  Minkowski sum of two point sets.
\end{defi}

\begin{eg}
Given $\b = (b_{1},\dots,b_{n}) \in \N^{n}$, the semigroup $\N \langle b_{1}, \dots, b_{n} \rangle$
consists of all non-negative integer combinations of the $b_{j}$'s. Its generating function is given by:
$$
f_{\b}(t) \, = \, \frac{1}{1-t^{b_{1}}} \. \oplus \, \dots \, \oplus \. \frac{1}{1-t^{b_{n}}}.
$$
Given such $\b \in \N^{n}$ and $a \in \N$, the $\textsc{KNAPSACK}$ problem asks if $a \in \supp(f_{\b})$.
\end{eg}

Below is the precise statement of Theorem~\ref{th:main_3}.

\begin{theo}\label{th:Minkowski_long}
Assume $\sharpP \not\subseteq \FPpo$.
Then there is an $s > 0$ and two sequences $\{a_{\ellR}\}_{\ellR > 0},\, \{b_{\ellR}\}_{\ellR > 0} \subset \GF_{1,s}$ such that
\begin{itemize}[leftmargin = 2em]
\item[\textup{a)}]
$\phi(a_{\ellR}) + \phi(b_{\ellR}) \le \polyin(\ellR)$.

\item[\textup{b)}] For every fixed $d$, the Minkowski sum $\,a_{\ellR} \oplus b_{\ellR}\,$ cannot be written as a short GF $h_{\ellR}$ in $\GF_{1,d}$ of length $\phi(h_{\ellR}) \le \polyin(\ellR)$.
\end{itemize}
\end{theo}

\begin{proof}
By Theorem~\ref{th:union_long}, there exists an $s > 0$, and for each $\ellR$ a subset
\[
S_{\ellR} = \{p_{\ellR,1},\dots,p_{\ellR,k_{\ellR}}\} \subset \GF_{1,s} \quad \text{with} \quad \sum \phi(p_{\ellR,i}) \le \polyin(\ellR)
\]
with the following property.
For every fixed $d$, the union $h_{\ellR} = p_{\ellR,i} \cup \dots \cup p_{\ellR,k_{\ellR}}$ cannot be written as a short GF of length $\polyin(\ellR)$ in $\GF_{1,d}$.
Define
\begin{equation}\label{eq:a_def}
a_{\ellR}(t,u) \,=\, \sum_{i=1}^{k_{\ellR}} p_{\ellR,i}(t) \. u^{i} \,\in\, \GF_{2,s}.
\end{equation}
and
\begin{equation}\label{eq:b_def}
b_{\ellR}(t,u) \,=\, \sum_{i=0}^{k_{\ellR}-1} t^{0} u^{i} = \frac{1 - u^{k_{\ellR}}}{1-u} \,\in\, \GF_{1,1} \subset \GF_{2,s}.
\end{equation}
Since $\sum \phi(p_{\ellR,i}) \le \polyin(\ellR)$, we also have $\phi(a_{\ellR}) + \phi(b_{\ellR}) \le \polyin(\ellR)$.
\smallskip

Consider the terms $t^{x} u^{k_{\ellR}}$ in the Minkowski sum $a_{\ellR} \oplus b_{\ellR}$. From~\eqref{eq:a_def} and~\eqref{eq:b_def}, we have:
\begin{equation*}
\big\{x : (x,k_{\ellR}) \in \supp(a_{\ellR} \oplus b_{\ellR}) \big\} \, = \, \bigcup_{i=1}^{k_{\ellR}} \supp(p_{\ellR,i})  \, = \, \supp(h_{\ellR}).
\end{equation*}
In other words, we have $[u^{k_{\ellR}}] (a_{\ellR} \oplus b_{\ellR}) (t,u) = h_{\ellR}(t)$.
Define
\begin{equation*}
g_{\ellR}(t,u) = \sum_{x \in \N} t^{x} u^{k_{\ellR}} =  \frac{u^{k_{\ellR}}}{1-t}.
\end{equation*}
Taking the intersection of $g_{\ellR}$ with $a_{\ellR} \oplus b_{\ellR}$, we get:
\begin{equation}\label{eq:extract_g}
\big[(a_{\ellR} \oplus b_{\ellR}) \star g_{\ellR}\big] (t,u) = u^{k_{\ellR}} \. h_{\ellR}(t).
\end{equation}

Now assume there is $d$ so that $a_{\ellR} \oplus b_{\ellR}$ can be written as $c_{\ellR} \in \GF_{2,d}$ with $\phi(c_{\ellR}) \le \polyin(\ellR)$.
By Theorem~\ref{th:BPGF2}, we can compute $h_{\ellR}$ by taking the Hadamard product $c_{\ellR} \star g_{\ellR}$ and substitute $u \gets 1$ in~\eqref{eq:extract_g}.
This would imply that $h_{\ellR}$ is a short GF of length $\polyin(\ellR)$ in the fixed class $\GF_{1,d+1}$, which contradicts our first statement on $h_{\ellR}$.

So the two sequences $\{a_{\ellR}\}_{\ellR > 0}$ and $\{b_{\ellR}\}_{\ellR > 0} \subset \GF_{2,s}$ do not have Minkowski sums of polynomial lengths.
Note that each $a_{\ellR}$ and $b_{\ellR}$ still has two variables.
By Lemma~\ref{lem:compress} part a), we can compress $a_{\ellR}, b_{\ellR}$ into single variable short GFs $\wt a_{\ellR}, \wt b_{\ellR} \in \GF_{1,s}$.
Note that compressing/decompression preserves Minkowski sum.
So $\wt a_{\ellR} \oplus \wt b_{\ellR}$ does not have polynomial length, because otherwise we can decompress it to get $a_{\ellR} \oplus b_{\ellR}$ of polynomial length.
\end{proof}

\bigskip

\section{Squares, primes, and short GFs} \label{sec:squares-primes}

\subsection{Short GFs and squares}
Recall the definition of the class $\Gzero$ from Section~\ref{sec:GF}.
We present a candidate for a language $\L \in \Ppo$ which is outside of $\Gzero$.
Let $\textsc{SQUARES}$ be the language consisting of all square numbers written in binary.
Then
\begin{equation}\label{eq:squares_l}
\textsc{SQUARES}_{\ellR} = \{k^{2} :\; k^{2} < 2^{\ellR}\}.\footnote{Strictly speaking, some numbers in $\textsc{SQUARES}_{\ellR}$ have less than $\ellR$ digits. However, we can always pad them with enough zeroes form a set of strings of the same length.}
\end{equation}

\begin{conj}\label{conj:squares_long}
$\textsc{SQUARES}$ is not in $\Gzero$.
\end{conj}

In other words, the conjecture says that for every fixed $s$, the segment $\textsc{SQUARES}_{\ellR}$ cannot be represented as $\supp(g_{\ellR})$ for a short GF $g_{\ellR} \in \GF_{1,s}$ of length $\phi(g_{\ellR}) \le \polyin(\ellR)$.
Note that this conjecture is free of complexity assumptions.
If true, Conjecture~\ref{conj:squares_long} shows unconditionally that $\Gzero \subsetneq \Ppo$, which implies $\Gzero \subsetneq \GH$.
We already know from Example~\ref{eg:squares} and Section~\ref{sec:GF} that $\textsc{SQUARES} \in \unique\PiG_{1} \subseteq \GH$. So $\text{SQUARES}$ should be a candidate that separates $\Gzero$ from $\unique\PiG_{1}$ according to this conjecture.

\smallskip
We begin with the following attractive result.

\begin{theo}\label{th:squares-factoring}
If Conjecture~\ref{conj:squares_long} is false, then \ts
$\textsc{INTEGER FACTORING} \in \BPP$.
\end{theo}

\begin{proof}
We build on an argument in Section 6 of~\cite{B2}.
Assume there is an $s>0$ so that for every $N = 2^{\ellR}$, we can write $\GFof\big(\text{SQUARES}_{\ellR};t\big) = g_{\ellR}(t)$,
%% \begin{equation}\label{eq:g_l}
%% \sum_{n^{2} < N} t^{n^{2}} \. = \. g_{\ellR}(t),
%% \end{equation}
where $g_{\ellR}(t)$ is a short GF in $\GF_{1,s}$ with $\phi(g) \le \polyin(\ellR)$.
Consider:
\[
h_{\ellR}(t) \. = \. g_{\ellR}(t)^{4} \. =  \. \Biggl( \sum_{n^{2} < N} t^{n^{2}} \Biggr)^4 \, = \. \sum_{k \ge 0 } a_{\ellR}(k) t^{k}\ts,
\]
where
\[
a_{\ellR}(k) \. = \. \# \Bigl\{ (n_{1},n_{2},n_{3},n_{4}) : n_{i}^{2}  < N, \; \sum n_{i}^{2} = k \Bigr\}\ts.
\]
In particular, if $k < N$, then $a_{\ellR}(k)$ is the number of ways to write $k$ as a sum of $4$ squares.
Since $g_{\ellR} \in \GF_{1,s}$, we have $h_{\ellR} = g_{\ellR}^4 \in \GF_{1,4s}$ and also $\phi(h) \le \polyin(\phi(g)) \le \polyin(\ellR)$.

Applying Proposition~\ref{prop:supp_is_P}, each coefficient $a_{\ellR}(k)$ can be computed in time $\polyin(\ellR)$.
By Jacobi's formula (see e.g.~\cite{HW}), we also have:
\[
a_{\ellR}(k)  \. = \. 8 \sum_{4 \nmid d,\. d | k} d  \quad \text{for} \quad k < N\ts.
\]
Here $d$ is a divisor of $k$ which is not a multiple of~$4$.
From this, we can compute in time $\polyin(\ellR)$ the sum of divisors $\sigma(k)$ for every $k < N = 2^{\ellR}$.
By a standard argument (see e.g.~\cite{BMS}), given $\sigma(k)$, a factorization of $k$
can be computed in probabilistic polynomial time~($\BPP$).
%Since $\BPP \subseteq \Ppo$ by Adleman's theorem (see eg.~\cite{MM,Pap}), we have $\textsc{INTEGER FACTORING} \in \Ppo$, as desired.
\end{proof}

\begin{theo} \label{t:squares-sharp}
If Conjecture~\ref{conj:squares_long} is false, then $\sharpP \subseteq \FPpo$.
\end{theo}

\begin{proof}[Proof of Theorem~\ref{t:squares-sharp}]
In~\cite{MA}, it is proved that the following problem is $\NP$-complete: Given $\al,\be,\ga \in \N$, decide whether there exists $x \in \N$ such that
\begin{equation}\label{eq:square_congruence}
0 \le x \le \ga \quad \text{and} \quad x^{2} \equiv \al \, (\text{mod}\; \be).
\end{equation}
The argument in~\cite{MA} actually gave bijection between the set of Boolean strings satisfying a $\textsc{3SAT}$ formula and the set of $x$ satisfying~\eqref{eq:square_congruence}.
Here $\al,\be$ and $\ga$ can be computed in polynomial time from the $\textsc{3SAT}$ formula.
Since counting the number of $\textsc{3SAT}$ solutions is $\sharpP$-complete, so is counting the number of solutions for~\eqref{eq:square_congruence}.

Now assume Conjecture~\ref{conj:squares_long} fails, then $\textsc{SQUARES} \in \Gzero$.
This means there is an $s > 0$ so that for every $\ellR > 0$ we can write  $\GFof\big(\text{SQUARES}_{\ellR};t\big) = g_{\ellR}(t)$ for some $g_{\ellR} \in \GF_{1,s}$ with $\phi(g_{\ellR}) \le \polyin(\ellR)$.
%% \begin{equation*}
%% \sum_{k^{2} < 2^{\ellR}} t^{k^{2}} = g_{\ellR}(t).
%% \end{equation*}
 %% (see \eqref{eq:g_l}).
Given $\al,\be,\ga \in \N$, we define:
\begin{equation*}
h(t) = \sum_{i=0}^{\ga^{2}} t^{i} = \frac{1-t^{\ga^{2} + 1}} {1-t} \quad \text{and} \quad k(t) = \sum_{x \equiv \al \, (\text{mod}\;\be)} t^{x} = \frac{t^{\al}}{1-t^{\be}}.
\end{equation*}
Let $\ellR = 2 \lceil \log\ga \rceil$.
The number of solutions for~\eqref{eq:square_congruence} can be counted by taking $g_{\ellR} \.\star\. h \.\star\. k$ and evaluate at $t = 1$, which are polynomial time operations by theorems~\ref{th:BPGF1} and~\ref{th:BPGF2}.
So the above $\sharpP$-complete problem can be solved by polynomial size circuits, which are provided by the $g_{\ellR}$ for different $\ellR$.
This implies $\sharpP \subseteq \FPpo$.
\end{proof}

By Theorem~\ref{th:Ppo_to_GF}, we can represent $\textsc{SQUARES}_{\ellR}$
as $\spec_{x}(B_{\ellR} \cpl \proj_{y}(f_{\ellR}))$ for some short GF $f_{\ellR}$ of length $\polyin(\ellR)$.
Conjecture~\ref{conj:squares_long} says that it is not possible to do so without using projections.
In the domain of $\Pr$ formulas, by Lemma~\ref{lem:Ppo_to_Pr}, we can represent $\textsc{SQUARES}_{\ellR}$ with a $\ex\for$-formula of length $\polyin(\ellR)$.
A similar question can be asked, i.e., are quantifiers necessary?
The following result shows that two quantifiers $\ex\for$ are necessary in Lemma~\ref{lem:Ppo_to_Pr}, already in the case of $\textsc{SQUARES}$.

\begin{prop}\label{prop:ex_not_enough}
$\textsc{SQUARES}_{\ellR}$ cannot be represented by an $\ex$-formula of length $\polyin(\ellR)$ in a fixed number of variables.
\end{prop}

\begin{proof}
By $\AP_{k}$ we mean a $k$-term arithmetic progression.
It is well known that $\textsc{SQUARES}$ does not contain any non-trivial $\AP_{4}$.  This was suggested
by Fermat in~1640 and proved by Euler in~1780 (see e.g.\ \cite[p.~115]{Weil}).
Also, the cardinality of $\textsc{SQUARES}_{\ellR}$ is super-polynomial in $\ellR$.
With these two observations, this proposition follows directly from the next theorem when $k=4$.
\end{proof}

\begin{theo}\label{th:ex_AP}
For every fixed $n$ and $k$, there exists a polynomial $P$ so that the following holds.
If an $\ex$-formula
\begin{equation}\label{eq:ex_formula}
\{ \ts x \;:\; \ex \y \in \zz^{n} \;\; \Phi(x,\y) \ts \}
\end{equation}
determines a set of cardinality at least $P(\phi(\Phi))$, then it must contain a non-trivial $\AP_k$.
\end{theo}

\begin{proof}
%% Assume $\L_{\ellR}$ can be represented as
%% \begin{equation}\label{eq:squares_ex}
%% x \in \L_{\ellR} \quad \iff \quad \ex \y \in \N^{n} \;:\; \Psi_{\ellR}(x,\y).
%% \end{equation}
  %% Here $n$ is some fixed dimension and $\Psi_{\ellR}$ is a quantifier free $\Pr$ expression with length $\phi(\Psi_{\ellR}) \le \polyin(\ellR)$.
By Proposition~\ref{prop:W}, we know that there is a constant $c = c(n)>0$ so that any quantifier free expression $\Phi$ in $n$ variables describes a disjoint union of $m$ polyhedra $P_{1},\dots,P_{m} \subseteq \R^{n+1}$ with $\ts m < \phi(\Phi)^{c}$.
So the formula~\eqref{eq:ex_formula} can be rewritten as:
\begin{equation}\label{eq:squares_P_i}
S \; = \; \Big\{ x \in \zz \; : \; \ex \y \in \zz^{n} \;\; \bigvee_{i=1}^{m} (x,\y) \in P_{i} \Big\}.
\end{equation}
%%Note that $\ts |\L_{\ellR}| = 2^{\Theta(\ellR)} \gg r$.
Let $\rq(t) = k^{n+1} t^{c}$. Assume that $|S| \ge \rq\bigl(\phi(\Phi)\bigr) > k^{n+1}m$.
Select any $(k^{n+1}m + 1)$ different integers from $S$.
By the pigeonhole principle, one of the polyhedra, say $P_{1}$, contains in its projection at least $k^{n+1}+1$ of these integers.
Denote those integers in the projection of $P_{1}$  by $x_{1},\dots,x_{s}$, where $s=k^{n+1}+1$.
%%Since each $x_{i} \in \L_{\ellR}$,
For every such $x_{i}$, there exists $\y_{i} \in \zz^{n}$ so that $(x_{i}, \y_{i}) \in P_{1}$.
So we have:
\begin{equation*}
(x_{1},\y_{1}),\dots,(x_{s},\y_{s}) \in P_{1} \cap \zz^{n+1}.
\end{equation*}
By the pigeonhole principle, two different pairs $(x_{i},\y_{i})$ and $(x_{j},\y_{j})$  have coordinates equal mod~$k$ pairwise. Since $P_{1}$ is convex, we also have
\begin{equation*}
(\la x_{i} + (1 - \la) x_{j},\; \la \y_{i} + (1 - \la) \y_{j}) \,\in\, P_{1} \cap \zz^{n+1}  \;, \quad \text{where} \; \la \in \big\{\tfrac{1}{k},\dots,\tfrac{k-1}{k}\big\}.
\end{equation*}
The above points project to $\la x_{i} + (1 - \la)x_{j}$.
By \eqref{eq:squares_P_i}, we get a non-trivial $\AP_{k+1}$:
\[
\Big( x_{i},\; \tfrac{k-1}{k}x_{i} + \tfrac{1}{k}x_{j}, \; \dots,\; \tfrac{1}{k}x_{j} + \tfrac{k-1}{k}x_{i},\; x_{j} \Big),
\]
a contradiction.
\end{proof}

\begin{rem}\label{rem:PA_strict}
Proposition~\ref{prop:ex_not_enough} combined with Lemma~\ref{lem:Ppo_to_Pr} implies that there is a sequence of formulas $\{ x : \ex y \. \for \z \; \Phi_{\ellR}(x,y,\z) \}$ of length $\polyin(\ellR)$ for which there are no equivalent formulas $\{x : \ex y \, \Psi_{\ellR}(x,y) \}$ of length $\polyin(\ellR)$.
This implies that the formulas $\{(x,y) : \for \z \; \Phi_{\ellR}(x,y,\z) \}$ have no equivalent quantifier free formulas in $x$ and $y$ of length $\polyin(\ellR)$.
Therefore, quantifier elimination in $\Pr$ necessarily increases the length of formulas by a super-polynomial factor, even in a bounded number of variables ($x,y \in \N,\, \z \in \N^{3}$).
%% This fact was proved by Gr\"{a}del (see \cite[Theorem~5.3]{G}) under the assumption $\NP \not\subseteq \Ppo$.
%% Here we proved it unconditionally.
\end{rem}

\begin{rem}\label{rem:both_not_enough}
From $\textsc{SQUARES}$, one can easily create another a language $\L \in \poly$ which $\L_{\ellR}$ be represented neither by $\for$ nor by $\ex$ formulas of length $\polyin(\ellR)$.
For $\ellR$ odd, we let $\L$ contain all squares between $2^{\ellR}$ and $2^{\ellR + 1}$. For $\ellR$ even, we let $\L$ contain all non-squares between $2^{\ellR}$ and $2^{\ellR+1}$.
It is clear that $\L \in \poly$.
The above argument shows that $\L_{\ellR}$ cannot be represented by $\ex$-formulas of length $\polyin(\ellR)$ when $\ellR$ is odd.
Under a negation, the same argument also works for $\for$-formulas when $\ellR$ is even.
We denote this language by $\textsc{SQUARES}'$.
This will be used in Section~\ref{sec:rel}.
\end{rem}

\subsection{Short GFs and arithmetic progressions}
Generalizing the above observation on sets with no arithmetic progressions, we suggest another conjecture on short GFs. Again, by $\AP_k$ we mean a \emph{$k$-term arithmetic progression}.

\begin{defi}\label{def:c-k}
Fix $c > 0$ and $k \ge 3$. A short GF $g$ is said to have the $(c,k)$-property
if either $|\supp(g)| < \phi(g)^{c}$ or $\supp(g)$ contains an $\AP_k$.
\end{defi}

\begin{conj}\label{conj:c-k}
For every $s$ and $k$, there exists $c > 0$ so that every short GF
$g(t) \in \GF_{1,s}$ has the $(c,k)$-property.
\end{conj}

%% \begin{eg}
%% Consider the set $\textsc{NSP}$ of \emph{near-square primes}, i.e., primes of the form $k^{2} + 1$.
%% Since \textsc{SQUARES} does not contain any $\AP_4$, the same holds for $\textsc{NSP}$.
%% The Prime Number Theorem suggests that the number of near-square primes between $1$ and $N$ is
%% $\Theta(\sqrt{N}/\log N)$, which is super-polynomial in $\log N$, see~\cite[$\S$6.IV]{Rib}.
%% %\footnote{This conjecture is sometimes attributed to E.~Landau (1912).}
%% If this conjectural asymptotics holds, then Conjecture~\ref{conj:c-k} implies
%% that the near-square primes up to $N$ do not have short GFs of size $\polyin(\log N)$.
%% On the other hand, if near-square primes have density only $\polyin(\log N)$,
%% %%then they can be written as short GFs of size $\polyin(\log(N))$.
%% %%Indeed, for any $N$, we can write a monomial sum:
%% then we can write a monomial sum:
%% \begin{equation*}
%% f_{N}(t) = \sum_{x \le N ,\, x \in \textsc{NSP}} t^{x}.
%% \end{equation*}
%% Such a sum $f_{N}$ is a short GF in the class $\GF_{1,0}\.$ with length
%% $O\bigl((\log N)^{c}\bigr)$ for some constant~$c$.
%% \end{eg}

% \begin{conj}[Landau]\label{conj:Landau}
% There are infinitely near-square primes.
% Moreover, the number of $\,p \in \textsc{NPS}$, $p \le N$ is $\Theta(\sqrt{n}/\log(N))$.
% \end{conj}

\begin{prop}\label{prop:c-k_implies_squares}
Conjecture~\ref{conj:c-k} implies Conjecture~\ref{conj:squares_long}.
\end{prop}
\begin{proof}
Assume Conjecture~\ref{conj:c-k} holds but Conjecture~\ref{conj:squares_long} fails, i.e., $\textsc{SQUARES} \in \Gzero$.
So there is an $s>0$ such that $\textsc{SQUARES}_{\ellR}$ can be represented as $\supp(g_{\ellR})$ with $g_{\ellR} \in \GF_{1,s}$ and $\phi(g_{\ellR}) \le \polyin(\ellR)$.
Conjecture~\ref{conj:c-k} applied to $s$ and $k=4$ gives us a $c > 0$ so that all $g \in \GF_{1,s}$ have the $(c,4)$-property.
We have $\supp(g_{\ellR}) = |\textsc{SQUARES}_{\ellR}| \gg \ellR^{c}$.
So if $\ellR$ is large enough, $g_{\ellR}$ contains an $\AP_4$.
This contradicts the fact that $\textsc{SQUARES}$ is $\AP_4$ free.
\end{proof}

%% Now, for the $\Pr$ formulas which contain some $\AP_k$, we have:

%% \begin{theo}\label{th:ex_AP}
%% For every fixed $n$ and $k$, there exists $c>0$ so that the following holds.
%% If a $\Pr$ formula
%% \begin{equation*}%%\label{eq:ex_formula}
%% x \;:\; \ex \y \in \N^{n} \;\; \Phi(x,\y)
%% \end{equation*}
%% determines a set with cardinality greater than $\phi(\Phi)^{c}$, then it contains an $\AP_k$.
%% \end{theo}

%% \begin{proof}
%% Similar to the proof of Proposition~\ref{prop:ex_not_enough}, with a $\AP_4$ replaced by a $\AP_k$.
%% \end{proof}

%% \begin{cor}\label{cor:super_polylog}
%% Consider a language $\L$ with super-polynomial density, i.e., for every fixed $c > 0$, we have $\ellR^{c} = o(|\L_{\ellR}|)$.
%% If there is $k \ge 3$ so that $\L$ contains only a bounded number of $\AP_k$, then for every fixed $n$, the segments $\L_{\ellR}$ cannot be represented by an $\ex$-formula $\Phi_{\ellR}$ $\phi(\Phi_{\ellR}) \le \polyin(\ellR)$.
%% \end{cor}

%% \begin{proof}
%% Discarding a large enough finite subset, we can make sure $\L$ is $\AP_k$--free.
%% Assume there is a sequence of $\Phi_{\ellR}$ with $\phi(\Phi_{\ellR}) \le \polyin(\ellR)$ which can represent $\L_{\ellR}$.
%% We have $|\L_{\ellR}|$ super-polynomial in $\ellR$.
%% So for every $c > 0$, we have $|\L_{\ellR}| > \phi(\Phi_{\ellR})^{c}$ for $\ellR$ large enough.
%% By Theorem~\ref{th:ex_AP}, we then have $\L_{\ellR}$ contains an $\AP_k$, a contradiction.
%% \end{proof}

\subsection{Short GFs and primes}
In a similar manner, we ask if primes can be represented by short GFs of polynomial length.
Let $\textsc{PRIMES}$ be the language consisting of all primes written in binary.
Then
\begin{equation}\label{eq:squares_l}
\textsc{PRIMES}_{\ellR} = \{p \text{ prime} :\; p < 2^{\ellR}\}.
\end{equation}

\begin{conj}\label{conj:primes_long}
%% For every fixed $s$, the segment $\textsc{PRIMES}_{\ellR}$ cannot be represented as $\supp(g_{\ellR})$ for a short GF $g_{\ellR} \in \GF_{1,s}$ with $\phi(g_{\ellR}) \le \polyin(\ellR)$ as $\ellR \to \infty$.
$\textsc{PRIMES}$ is not in $\Gzero$.
\end{conj}

In other words, the conjecture says that for every fixed $s$, the segment $\textsc{PRIMES}_{\ellR}$ cannot be represented as $\supp(g_{\ellR})$ for a short GF $g_{\ellR} \in \GF_{1,s}$ of length $\phi(g_{\ellR}) \le \polyin(\ellR)$.
This conjecture, if true, would also show $\Gzero \subsetneq \Ppo$ unconditionally.

\begin{prop}
Let $\pi(n)$ be the number of primes between $1$ and $n$.
If Conjecture~\ref{conj:primes_long} is false then $\pi(n)$ can be computed by circuits of size $\polyin(\log n)$.
\end{prop}

\begin{proof}
Assume Conjecture~\ref{conj:primes_long} is false, i.e., there is an $s > 0$ so that for every $\ellR > 0$ we have $\GFof\big(\textsc{PRIMES}_{\ellR};t\big) \. = \. g_{\ellR}(t)$,
%% \begin{equation*}
%% \sum_{x \in \textsc{PRIMES}_{\ellR}} \ts t^{x} \. = \. g_{\ellR}(t)\ts,
%% \end{equation*}
where $g_{\ellR} \in \GF_{1,s}$ and $\phi(g_{\ellR}) \le \polyin(\ellR)$.
Given $n < 2^{\ellR}$, we have:
\begin{equation*}
%%  \sum_{\substack{x \in \textsc{PRIMES}_{\ellR}\\ x \le n}} t^{x}
  \GFof\big(\textsc{PRIMES}_{\ellR} \cap [0,n];\, t\big) \;=\; g_{\ellR}(t) \, \star \, \frac{1 - t^{n+1}}{1 - t} \. =\. h_{n}(t)\ts.
\end{equation*}
By Theorem~\ref{th:BPGF2}, we can compute $h_{n}$ in time $\polyin(\ellR)$.
Substituting $t \gets 1$, we get $\pi(n)$.
\end{proof}

\begin{rem}
In \cite{LO}, using strong analytic tools, Lagarias and Odlyzko gave an
algorithm to compute $\pi(n)$ in time $O(n^{1/2 + \ep})$, which is
exponential in $\log n$.  If Conjecture~\ref{conj:primes_long} is false,
then for each~$\ellR$, a far better $\polyin(\ellR)$ algorithm exists for
computing $\pi(n)$ for all $n < 2^{\ellR}$.
\end{rem}

\bigskip

\section{Relative complexity of short GFs}\label{sec:rel}
In this section, we compare short GFs with $\Pr$ formulas with one quantifier.
We refer back to Section~\ref{sec:pres} for the definition of $\Pr$ formulas.

\subsection{$\Pr$ complexity classes}
The most basic PA formulas contain no quantifiers, i.e., only a Boolean combination of inequalities.

\begin{defi}
  The class $\SigmaPA_{0} = \PiPA_{0}$ consists of languages definable by quantifier free PA formulas of polynomial lengths.
  In other words, a language $\L$ is in $\SigmaPA_{0}$ if for every $\ellR > 0$, there is a quantifier free PA expression $\Phi_{\ellR}(x)$ of length $\phi(\Phi) \le \polyin_{\L}(\ellR)$ so that:
  \begin{equation*}
x \in \L_{\ellR} \quad \iff \quad \Phi_{\ellR}(x).
\end{equation*}
\end{defi}

By Proposition~\ref{prop:W}, $\L \in \SigmaPA_{0}$ if and only if every initial segment $\L_{\ellR}$ is a union of polynomially many intervals in $\NN$.
By Theorem~\ref{th:W}, we have $\SigmaPA_{0} \subset \Gzero$.

\begin{eg}
The language $\textsc{EVEN}$ of even integers is not in $\SigmaPA_{0}$. However, $\textsc{EVEN} \in \Gzero$, because:
\begin{equation*}
\sum_{x \in \textsc{EVEN}_{\ellR}} t^{x} = t^{0} + t^{2} + \dots + t^{2^{\ellR}-2} = \frac{1-t^{2^{\ellR}}}{1-t^{2}}.
\end{equation*}
So we conclude that $\SigmaPA_{0} \subsetneq \Gzero$.
\end{eg}

\begin{defi}
The class $\SigmaPA_{1}$ consists of languages definable by $\ex$-formulas of polynomial lengths.
In other words, $\L \in \SigmaPA_{1}$ if there is an $n$ so that for every $\ellR > 0$, we can represent
\begin{equation*}
x \in \L_{\ellR} \quad \iff \quad \ex \y \in \N^{n} \;\; \Phi_{\ellR}(x,\y),
\end{equation*}
where $\Phi_{\ellR}(x,\y)$ is a quantifier-free PA expression of length $\phi(\Phi_{\ellR}) = \polyin_{\L}(\ellR)$.
The class $\PiPA_{1}$ is defined similarly, but with $\for$-formulas.
In other words, $\L \in \PiPA_{1}$ if and only if $\lnot\L \in \SigmaPA_{1}$.
\end{defi}

%% Any formula $F$ in $\for^{n}$ or $\ex^{n}$ defines a set of natural numbers $x$ satisfying it. We also denote this set by $F$, and its cardinality by $|F|$.
%%Moreover, every formula in $\for^{n}$ is the negation of a formula in $\ex^{n}$, and vice versa.
%% This is not the same as the binary length of $F$, which is denoted by $\phi(F)$.

\begin{conj}\label{conj:short_in_ex}
$\Gzero \subseteq \SigmaPA_{1} \cap \PiPA_{1}$.
%% For every $s$, there exists $n = n(s)$ and $d =d(s)$ so that the following holds. For every $f \in \GF_{1,s}$, there exists $F \in \ex^{n}$ with $F = \supp(f)$ and $\phi(F) < \phi(f)^{d}$.
\end{conj}

To rephrase, this conjecture says that for every fixed $s$, there is an $n = n(s)$ so that every $g \in \GF_{1,s}$ of finite support has an $\ex$-formula representation:
\begin{equation}\label{eq:SigmaPA1}
G = \{x : \ex \y \in \N^{n} \;\; \Phi(x,\y) \}, \quad \GFof(G;t) = g(t) \quad \text{and} \quad \phi(\Phi) \le \polyin(\phi(g)).
\end{equation}
Note that it would be enough to show $\Gzero \subseteq \SigmaPA_{1}$, because $\Gzero$ is closed under taking complement of short GFs.

%% \begin{prop}
%% If Conjecture~\ref{conj:short_in_ex} holds, we also have $\GF_{1,s} \polysubset \ex^{n} \cap \for^{n}$.
%% \end{prop}

%% \begin{proof}
%% Observe that $\GF_{1,s}$ is closed under negation, i.e., for every $f \in \GF_{1,s}$, we also have $\lnot f \in \GF_{1,s}$, where $\lnot f(t) = \frac{1}{1-t} - f(t)$.
%% Conjecture~\ref{conj:short_in_ex} asserts that $\GF_{1,s} \polysubset \ex^{n}$.
%% Since $\lnot f \in \GF_{1,s}$, we have $\lnot f \in_{\text{poly}} \ex^{n}$.
%% Taking negation, we also have $f \in_{\text{poly}} \for^{n}$.
%% So $f \in_{\text{poly}} \ex^{n} \cap \for^{n}$ for every $f \in \GF_{1,s}$.
%% %%Taking negation, $\GF_{1,s} = \lnot \GF_{1,s} \polysubset \lnot\ex^{n} = \for^{n}$.
%% This implies $\GF_{1,s} \polysubset \ex^{n} \cap \for^{n}$.
%% \end{proof}

\begin{prop}\label{prop:best-case}
Conjecture~\ref{conj:short_in_ex} implies Conjecture~\ref{conj:c-k}, which implies Conjecture~\ref{conj:squares_long}.
\end{prop}

\begin{proof}
  Assume Conjecture~\ref{conj:short_in_ex} holds.
Then for every fixed $s$, we have $n = n(s)$ for which every $g \in \GF_{1,s}$
has an $\ex$-formula representation~\eqref{eq:SigmaPA1}.
The last condition means there is a constant $d=d(s)$ such that $\phi(\Phi) < \phi(g)^d$.
By Theorem~\ref{th:ex_AP}, there exists $\ga = \ga(n,k) > 0$ so that $G$ contains an $\AP_k$ whenever $|G| > \phi(\Phi)^{\ga}$.
So if $|\supp(g)| \ge \phi(g)^{\ga d}$ then $|G| = |\supp(g)| \ge \phi(g)^{\ga d} > \phi(\Phi)^{\ga}$, which implies that $G$ contains an $\AP_k$.
So $c = \ga d$ satisfies Conjecture~\ref{conj:c-k}, which should depend only on $s$ and $k$.
By Proposition~\ref{prop:c-k_implies_squares}, Conjecture~\ref{conj:c-k} implies Conjecture~\ref{conj:squares_long}.
\end{proof}

The picture below illustrates the relative relations between short GFs and $\Pr$ formulas, assuming Conjecture~\ref{conj:short_in_ex}:

\vspace{-1em}
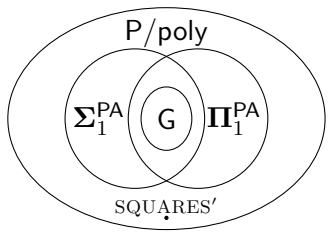
\begin{figure}[!h]
\centering
\begin{tikzpicture}[scale=1.2]

%%\draw (0,0) ellipse (70pt and 50pt);
\draw (0,0) ellipse (50pt and 35pt);
\draw (-0.35,0) circle (22pt);
\draw (0.35,0) circle (22pt);
\draw (0,0) ellipse (8pt and 10pt);

%%\node (NP) at (0,1.45) {\large$\NP\po = \ex\for$};
\node (ex) at (-0.75,0) {$\SigmaPA_{1}$};
\node (for) at (0.75,0) {$\PiPA_{1}$};
\node (GF) at (0,0) {$\Gzero$};
\node (P) at (0,0.95) {$\poly\po$};
\node (SQUARES') at (0,-0.98) {\tiny$\textsc{SQUARES}'$};
\draw[black,fill=black] (0,-1.1) circle (.5pt);

\end{tikzpicture}
\caption{Short GFs vs. PA formulas.
%%Here $\SigmaG_{2} = \NP\po$.
Here $\textsc{SQUARES}'$ is the language defined in Remark~\ref{rem:both_not_enough}.
}
\label{pic:complexity}
\end{figure}

%% Denote by $\SigmaPA_{1}$ ($\PiPA_{1}$) the union of all classes $\ex^{n}$ ($\for^{n}$) for all $n > 0$.

One can of course define analogues of $\SigmaPA_{1}$ and $\PiPA_{1}$ with more alternating quantifiers.
But it turns out that $\SigmaPA_{k+1} = \SigmaG_{k+1} = \SigmaP_{k}\po$ for every $k \ge 1$. This was implicit in Lemma~\ref{lem:PH_to_Pr} and theorems~\ref{th:PH_to_GF},~\ref{th:GH_PHpo}.
For the sake of completeness, we call the hierarchy of all classes $\SigmaPA_{k}$ and $\PiPA_{k}$ as $\GPA$. Obviously $\GPA = \GH = \PH\po$.

\medskip

\subsection{Complexity classes diagram}
The following diagram summarizes various complexity classes that appeared in this
paper and their relationships.  An arrow $X \to Y$ indicates $X \subseteq Y$.
Known strict subset relations are decorated with $\neq$.
Dashed arrows and segments denotes conjectural relationships.

\[
\begin{tikzcd}[row sep=scriptsize]
\GPA \arrow[r, white, "\text{\large =}" {black,description}]  &\GH \arrow[r, white, "\text{\large =}" {black,description}] &\PH\po  \\
\SigmaPA_{3} \arrow[u,dotted] \arrow[r, white, "\text{\large =}" {black,description}]  &\SigmaG_{3} \arrow[u,dotted]  \arrow[r, white, "\text{\large =}" {black,description}]  &\SigmaP_{2}\po \arrow[u,dotted]\\
\SigmaPA_{2} \arrow[u] \arrow[r, white, "\text{\large =}" {black,description}] &\SigmaG_{2} \arrow[u] \arrow[r, white, "~~\text{\large =}" {black,description}] & \SigmaP_{1}\po \arrow[u]\\
~ &\unique\PiG_{1} \arrow[u] \arrow[r, white, "~~\text{\large =}" {black,description}] &\unique\poly\po \arrow[u]\\
~ &~ &\Ppo \arrow[u] &\hspace{-3.2em}\ni {\small\textsc{SQUARES}} \arrow[ddlll, dash, bend left=12, "\not\ni"'] \arrow[dddll, dashed, red, dash, bend left=12, "\not\ni?"]\\
~ &\SigmaG_{1} \arrow[uu] \arrow[ur]  &~\\
\SigmaPA_{1} \arrow[uuuu, "\text{\rotatebox{90}{$\neq$}}"] \arrow[ur] \\
~ &G \arrow[uu] \arrow[ul, dashed, red, "?"']\\
\SigmaPA_{0} \arrow[uu, "\text{\rotatebox{90}{$\neq$}}"] \arrow[ur, "\text{\rotatebox{29}{$\neq$}}"']
\end{tikzcd}
\]

\bigskip

\nin
$
\begin{matrix*}[l]
\text{$\SigmaPA_{k+1}=\SigmaG_{k+1}=\SigmaP_{k}\po$, $k \ge 1$: sections~\ref{sec:GF}, \ref{sec:rel}.}  & & &  \text{$\textsc{SQUARES} \stackrel{?}{\notin} \Gzero$: Conjecture~\ref{conj:squares_long}.}\vspace{.2em}\\

\text{$\unique\PiG_{1}=\unique\Ppo$: Remark~\ref{rem:UP}.} & & & \text{$\textsc{SQUARES} \notin \SigmaPA_{1}$: Proposition~\ref{prop:ex_not_enough}.}\vspace{.2em}\\

\text{$\SigmaG_{1} \subseteq \Ppo$: Proposition~\ref{prop:supp_proj_is_P}.} & & & \text{$\SigmaPA_{0} \subsetneq \SigmaPA_{1} \subsetneq \SigmaPA_{2}$: Remark~\ref{rem:PA_strict}.}\\

 & & &\text{$\SigmaPA_{0} \subsetneq \Gzero \stackrel{?}{\subseteq} \SigmaPA_{1}$: Section~\ref{sec:rel}.}
\end{matrix*}
$

\bigskip

\section{Proof of Lemma~\ref{lem:compress}} \label{s:lemma-proof}

Let $\ov \x = (\x_{1}, \dots, \x_{k})$ be the array of multi-variables of dimension $n_{1}, \dots, n_{k}$.
We first prove the result when $k=1$, i.e., when  $\ov \x = \x_{1}$, $g(\ov \t) = \sum \t_{1}^{\x_{1}}$ and $f(u) = \sum u_{1}^{z_{1}}$.
For convenience, we denote $\t_{1},\.\x_{1},\.u_{1},\.z_{1}$ by $\t,\.\x,\.u$ and $z$ respectively.
Also denote by $n$ the dimension of the multi-variable $\x$.
So $g(\t) = \sum \t^{\x}$ and
\begin{equation*}
\tau_{N}(\x) = x_{1} + N x_{2} + \dots + N^{n-1} x_{n}.
\end{equation*}

{\bf Part a)} Assume we are given $g \in \GF_{n,s}$.
By Theorem~\ref{th:BPGF1}, we can find the norm $N$ of $g$ in time $\polyin(\phi(g))$.
By rounding $N$ to the next power of $2$,  we still have
$\ts \log N \le \polyin(\phi(g))\ts$ and $\ts\supp(g) \subseteq [0,N)^{n}$.
Let $N = 2^{\ellR}$.
We define $f(u)$ be the specialization of $g(\t)$ under the following substitutions:
\[
t_{1} \gets u, \; t_{2} \gets u^{N}, \ldots, \; t_{n} \gets u^{N^{n-1}},
\]
so that
\[
\t^{\x} \. = \. u^{x_{1} + N x_{2} + \ldots + N^{n-1} x_{n}} \. = \. u^{\tau_{N}(\x)}.
\]
Clearly, we have:
\begin{equation*}
\supp(f) \,=\, \tau_{N}(\supp(g)).
\end{equation*}
By Theorem~\ref{th:BPGF1}, polynomial substitutions can be performed in polynomial time and gives $f$ as a short GF in $\GF_{1,s}$ with $\phi(f) \le \polyin(\phi(g))$.
This proves part a).
\smallskip

{\bf Part b)}
Given two power series $A(\t) = \sum \al_{\x} \t^{\x} \in \GF_{n,\ts p} \, , \, B(t) = \sum \be_{x} t^{x} \in \GF_{1,\ts q}$ and a linear map $\tau : \Z^{n} \to \Z$,
we define their \emph{$\tau$-Hadamard product} as
\begin{equation}\label{tauHadDef}
C(\t) = A(\t)  \tauHad B(t) \,\coloneqq\, \sum \al_{\x} \be_{\tau(\x)} \t^{\x}\..
\end{equation}
%%Here, we allow $A(\t)$ and $B(t)$ to have general coefficients, i.e., not necessarily $0/1$.

Now assume $f(u) = \sum u^{z} \in \GF_{1,s}$, $N = 2^{\ellR}$, and $\supp(f) \subseteq [0,N)^{n}$.
From the above definition, it is clear that such a $g(\t)$ satisfying~\eqref{eq:compress} can be obtained as:
\begin{equation}\label{eq:decompress}
g(\t) = a(\t) \star_{\tau_{N}} f(t),
\end{equation}
where
\[
a(\t) \. = \, \sum_{\x \in [0,N)^{n} } \t^{\x} \, = \, \frac{1-t_{1}^{N}}{1-t_{1}} \. \cdots \.\frac{1-t_{n}^{N}}{1-t_{n}}.
\]
with $a \in \GF_{n,n}$ and $\phi(a) \le \polyin(\log N)$.

Here the map $\tau_{N}$ is from Definition~\ref{def:tau_map}.
So it is enough to show that the $\tau$-Hadamard product of two short GFs is a short GF of polynomial length.
 The proof follows Barvinok's argument in \cite{B2} (see also lemmas 3.4 and~3.6 in \cite{BW}).
 First, notice that the $\tau$-Hadamard product is bilinear in $A(\t)$ and $B(t)$.
Therefore, we only need to show that $C(\t)$ is a short GF when $A(\t)$ and $B(\t)$ have only $1$ term each, i.e., when:
\begin{equation}\label{singleterm}
A(\t) \, = \, \frac{\t^{\a}}{\prod_{i=1}^{p} (1 - \t^{\b_{i}})} \ \quad \text{and} \quad B(t) \, = \, \frac{t^{c}}{\prod_{j=1}^{q} (1-t^{d_{j}})}\..
\end{equation}

Consider an (unbounded) polyhedron $\ts P \subset \R^{p + q}\ts$ with coordinates $\ts (\zeta_{1}, \dots, \zeta_{p}, \xi_{1}, \dots, \xi_{q})$, defined as:
\begin{equation}\label{HadamardPolytope}
P \coloneqq
\begin{Bmatrix}
\zeta_{1},\dots,\zeta_{p},\,\xi_{1},\dots,\xi_{q} &\ge &0 \\
\tau(\a + \zeta_{1} \b_{1} + \dots + \zeta_{p} \b_{p}) &= &c + \xi_{1} d_{1} + \dots + \xi_{q} d_{q}
\end{Bmatrix}.
\end{equation}
By Theorem~\ref{th:Barvinok}, we can write a short GF for $P \cap \Z^{p+q}$:
\begin{equation}\label{diagonal}
D(\w,\v) \coloneqq \sum_{(\zzeta,\xxi) \in P} \w^{\zzeta} \v^{\xxi} = \sum_{(\zzeta,\xxi) \in P} (w_{1})^{\zeta_{1}} \dots (w_{p})^{\zeta_{p}} (v_{1})^{\xi_{1}} \dots (v_{q})^{\xi_{q}}.
\end{equation}
Furthermore, we have $D \in \GF_{p+q,\ts p+q}\.$.
By \eqref{singleterm}, the expansions of $A(\t)$ and $B(t)$ are:
\begin{equation}\label{expansion}
A(\t) = \sum_{\zzeta \ge 0} \t^{\a + \zeta_{1} \b_{1} + \dots + \zeta_{p} \b_{p}} \quad \text{and} \quad B(t) = \sum_{\xxi \ge 0} t^{c + \xi_{1} d_{1} + \dots + \xi_{q} d_{q}}.
\end{equation}
We substitute:
\[
w_{1} \gets \t^{\b_{1}} , \dots , w_{p} \gets \t^{\b_{p}}, v_{1} \gets 1 , \dots , v_{q} \gets 1.
\]
By~\eqref{HadamardPolytope},~\eqref{diagonal} and~\eqref{expansion}, we get:
\[
\t^{\a} D(\t^{\b_{1}}, \dots , \t^{\b_{p}}, 1, \dots , 1) = \sum_{(\zzeta,\xxi) \in P} \t^{\a + \zeta_{1} \b_{1} + \dots + \zeta_{p} \b_{p}} = A(\t) \tauHad B(t) = C(\t).
\]
By Theorem~\ref{th:BPGF1}, substitution can be done in polynomial time, and results in a short GF $C(\t)$ of index at most $p+q$.
Hence, we have $C(\t) \in \GF_{n,p+q}$ and $\phi(C) \le \polyin(\phi(A) + \phi(B))$.
Note that by taking the $\tau$-Hadamard product, the index of $C$ is increased to $p+q$.
This pushes the index of $g$ in~\eqref{eq:decompress} to $n+s$.
So we do not get back exactly the index $s$ for $g$.
But $n+s$ is still a constant, and $g$ is still a short GF in a fixed class $\GF_{n,n+s}$.
\smallskip

This completes the proof for the case $k=1$. The general case can be handled similarly.

\bigskip

\section{Final remarks and open problems} \label{sec:fin-rem}

\subsection{}  \label{ss:finrem-bar}
As we mentioned in the introduction, much of this work is
motivated by Barvinok's program implicit in his writing.  Specifically,
we were inspired by the following quote:

\smallskip

\qquad ``It seems hard to prove that a particular finite, but large, set $S\ssu \zz^d$ does not

\qquad  admit a short rational generating function: if a particular candidate expression

\qquad  for $f_S(\bx)$ is not short, one can argue that we have not searched hard enough

\qquad and that there is another, better candidate.''~\cite{B2}

\smallskip

\nin
In fact, this paper originally began as a followup on~\cite{NP1}, aiming
to explain why the technology of short GFs was unable to derive the
Barvinok--Woods theorem (Theorem~\ref{th:BW}) (cf.~\cite{NP1}).
Our theorems~\ref{th:main_2} and~\ref{th:main_1} are strong versions of
this claim.

Let us also mention  Theorem~1.1 and Corollary~1.11
in~\cite{NP-kannan} which have similar setup of unions and projections
of polyhedra, and give strong algorithmic extensions
of Woods's theorem (Theorem~\ref{th:GF_NP_hard}).

Finally, our most recent results in~\cite{NP-hard} say that Presburger Arithmetic with a bounded number of variables and inequalities is
complete for every level in $\PH$, which suggests an even deeper
obstacle to taking unions and projections.  We have yet to fully explore
the implications of this result which go beyond the scope of this paper.

\subsection{}  \label{ss:finrem-AP}
In notations of the introduction, a short GF $f_S(t)$ of a set $S \ssu \nn$ can be viewed as a
presentation of $S$ by an alternating sum of generalized
($k$-dimensional) arithmetic progressions.  As such, there are many
connections between short GFs and Arithmetic Combinatorics, which
are yet to be explored (cf.~\cite{TV}).  For example, when $k=1$,
taking the positive part of these arithmetic progressions corresponds
to variants of Erd\H{o}s's \emph{covering systems} which received much attention
in recent years (see~\cite{Guy,Hough}).

Conjecture~\ref{conj:squares-intro} has an especially classical feel with
its claim that squares and (generalized) arithmetic progression are incompatible.
There are of course  both classical and recent works on squares in arithmetic
progressions, but no known results seem strong enough to apply in this case
(see~\cite{BGP,Sze,Weil}).

%
% In particular, this gives a polynomial time algorithm for the
% \emph{Frobenius problem} on counting integers not representable
% as sums of integers from a given  finite set~$A$.
% See~\cite{RA} for the detailed study of the Frobenius problem.

\subsection{} \label{ss:finrem-converse}
There are two ways to think of the results in this paper.
First and foremost, they provide a very strong evidence in favor of
non-polynomiality of projections and other operations with short GFs.
In the opposite direction, the apparent connection to arithmetic
progressions and a plethora of both analytic and combinatorial tools for
working with them suggest a possibility of some lower bounds.

We would like to caution the reader. Initially we were rather
optimistic about removing complexity assumptions in
Theorem~\ref{th:main_1} by finding a direct proof of
Conjecture~\ref{conj:squares_long} or some other similar
lower bound. However, Proposition~\ref{prop:partial_converse} and Remark~\ref{rem:hard}
seem to suggest that this might be rather difficult.
A sufficiently strong argument that shows $\Gzero \subsetneq \GH$ could potentially show
$ \unique\Ppo = \unique\PiG_{1} \subsetneq \GH$, which implies $\sharpP \not\subseteq \FPpo$,
an important open problem (see~$\S$\ref{ss:finrem-factoring} below).

On the other hand, the two lowest level $\Gzero$ and $\SigmaG_{1}$ in $\GH$ seems to behave quite differently from higher ones.
So an elementary approach to prove $\Gzero  \subsetneq \GH$ is not completely ruled out.

% Either way, the relative complexity of short GFs seem to be of interest in
% its own right even without applications to Number Theory and Discrete Geometry.
% We intend to continue pursuing this subject in the future.

\subsection{}\label{ss:finrem-rel}
The idea of Section~\ref{sec:rel} is to characterize all short~GFs.  Roughly,
Conjecture~\ref{conj:short_in_ex} says that every short~GF is the projection of
a union of polynomially many polyhedra of bounded dimension.
This can viewed as a converse of the Barvinok--Woods theorem (Theorem~\ref{th:BW}).

Conjecture~\ref{conj:short_in_ex} is possibly a wishful thinking.
Unfortunately, its validity is hard to judge since we have so few
explicit constructions of short~GFs other than projections of
integer points in polyhedra. If true,
Proposition~\ref{prop:best-case} implies Conjecture~\ref{conj:squares-intro}
and removes the complexity assumptions from all theorems in the
introduction.  Moreover, it implies exponential lower bounds on
the length of short~GF for squares, projections and other theorems
in the introduction.\footnote{In the chain of reductions, the exponential factor
appears in the proof of Proposition~\ref{prop:c-k_implies_squares}.}
These are the same bounds the \emph{exponential time hypothesis}
(ETH) implies.

\subsection{}\label{ss:finrem-factoring}
It is worth comparing theorems~\ref{th:squares-factoring}
and~\ref{t:squares-sharp} from the computational complexity
point of view.
Technically speaking, these two results are not comparable.
However, one is weaker than the other in the relative sense, as follows.

%% First note that $\BPP \subseteq \Ppo$ by Adleman's theorem
%% (see e.g.~\cite{MM,Pap}).

Recall that \textsc{INTEGER FACTORING} $\in\NP \cap\coNP$.
While proving it to be in $\BPP$ would
be a very strong result beyond the current state of art, it would not directly lead to
a collapse of $\PH$.  In fact, the experts seem
to be split on whether \textsc{INTEGER FACTORING} is in~$\Pp$, all the while espousing
a deep-seated belief that $\Pp=\BPP$, thus further muddling the subject
(see~\cite{Aar,Gas}).  In summary, Theorem~\ref{th:squares-factoring}
gives a relatively weak evidence in favor of Conjecture~\ref{conj:squares_long}.

On the other hand, $\SP$-complete oracles are very powerful by Toda's theorem,
and thus very unlikely to be in~$\FPpo$.
As mentioned in Remark~\ref{rem:strong_collapse}, $\SP \subseteq \FPpo$ would lead to a collapse of $\PH$ the second level.
In other words, Theorem~\ref{t:squares-sharp} gives a very strong evidence in favor of Conjecture~\ref{conj:squares_long}.

\bigskip

\vskip.56cm

\subsection*{Acknowledgements}
We are grateful to Matthias Aschenbrenner, Sasha Barvinok, Boris Bukh,
Terry Tao, Kevin Woods, Josh Zahl and the anonymous referees for many helpful remarks on the subject.
We are also thankful to Joshua Grochow, Emil Jer\'abek for help
with complexity questions.
The second author was partially
supported by the~NSF.

%\medskip

%\bigskip

\newpage

{\footnotesize

%}

\newpage

}\end{document}